\documentclass[12pt]{extarticle}
  \usepackage{booktabs}
\usepackage{amsmath, amsthm, amssymb, hyperref}
\usepackage{graphicx}
\usepackage[all]{xypic}
\usepackage{verbatim}
\usepackage{cleveref}
\usepackage{todonotes}
\usepackage{algorithm}
\usepackage{algpseudocode}
\usepackage{cite}
\usepackage{tikz}
\usepackage{xcolor}
\oddsidemargin \evensidemargin
\graphicspath{ {./figures_small/} }

\newtheorem{theorem}{Theorem}
\numberwithin{theorem}{section}
\newtheorem{proposition}[theorem]{Proposition}
\newtheorem{lemma}[theorem]{Lemma}
\newtheorem{corollary}[theorem]{Corollary}
\newtheorem{definition}[theorem]{Definition}

\newtheorem{remark}[theorem]{Remark}
\newtheorem{example}[theorem]{Example}
\newtheorem{conjecture}[theorem]{Conjecture}
\newcommand{\PP}{\mathbb{P}}
\newcommand{\RR}{\mathbb{R}}
\newcommand{\QQ}{\mathbb{Q}}
\newcommand{\CC}{\mathbb{C}}
\newcommand{\ZZ}{\mathbb{Z}}
\newcommand{\SymRR}{\mathcal{S}}
\newcommand{\id}{I}

\DeclareMathOperator{\rank}{rank}
\DeclareMathOperator{\trace}{trace}
\DeclareMathOperator{\codim}{codim}
\DeclareMathOperator*{\argmax}{arg\,max}

\newcommand{\bigA}{\mathcal{A}}
\newcommand{\bigC}{\mathcal{C}}

\date{}

\title{\textbf{The Geometry of SDP-Exactness
\\ in Quadratic Optimization}}

\author{Diego Cifuentes, Corey Harris and Bernd Sturmfels}

\begin{document}

\maketitle

  \begin{abstract} \noindent  Consider the problem of
  minimizing a quadratic objective subject to quadratic \mbox{equations}. We study the
  semialgebraic region of objective functions for which this problem is solved by its
  semidefinite relaxation. For the Euclidean distance problem, this is a
    bundle of spectrahedral shadows surrounding the given variety.
    We characterize the algebraic boundary of this region
   and we derive a formula for its degree.
   \end{abstract}

\section{Introduction}

We study a family of \emph{quadratic optimization} problems with varying cost function:
\begin{equation}
\label{eq:qp1}
\min_{x \in \RR^n} \,\, g(x) \quad \hbox{subject to} \quad
f_1(x) = f_2(x) = \cdots = f_m(x) = 0,
\end{equation}
where ${\bf f} = (f_1,\ldots,f_m)$ is a fixed tuple
of elements in  the space $\RR[x]_{\leq 2} \simeq \RR^{\binom{n+2}{2}}$ of
polynomials of degree two in $x = (x_1,\ldots,x_n)$.
The problem~\eqref{eq:qp1} is hard, but {semidefinite programming} (SDP) offers a tractable approach.
Indeed, there is a hierarchy of SDP relaxations of~\eqref{eq:qp1}; see, e.g.,~\cite{BPT,Lasserre2010,Lasserre2001}.
In this paper we focus on the first and simplest of these relaxations, also known as Shor relaxation~\cite{Shor1987}.
We are interested in the set defined by the Shor relaxation:
\begin{equation*}
\mathcal{R}_{\bf f} \,\, = \,\,
\bigl\{\, \,g \in \RR[x]_{\leq 2} \,: \, \text{the problem \eqref{eq:qp1} is solved exactly
by its SDP relaxation} \,\bigr\}.
\end{equation*}
We call $\mathcal{R}_{\bf f}$ the \emph{SDP-exact region} of the tuple
$\,{\bf f} = (f_1,\ldots,f_m)$.
We will slightly change this definition in \Cref{s:3} by further imposing strict complementarity.
 This will lead to an explicit description of~$\mathcal{R}_{\bf f}$ as a semialgebraic set in $\RR[x]_{\leq 2}\simeq \RR^{\binom{n+2}{2}}$.
We refer to \Cref{def:rank1sol2}.

The quadratic cost function that motivated this article is the squared distance
to a given point $u \in \RR^n$. In symbols, $\,g_u(x) = || x - u ||^2$.
Here \eqref{eq:qp1} is the \emph{Euclidean Distance} (ED) problem (cf.~\cite{DHOST}) for the variety
$\,V_{\bf f} = \{ x \in \RR^n \,: \,f_1(x) = \cdots = f_m(x) = 0 \}$.
By restricting $\mathcal{R}_{\bf f}$ to the space of cost functions $g_u$,
we obtain a semialgebraic set in $\RR^n$.
This is the SDP-exact region for the ED~problem, denoted $\mathcal{R}_{\bf f}^{\it ed}$, which was investigated in~\cite{CAPT}.

\begin{figure}[htb]
 \centering
 \includegraphics[height=180pt]{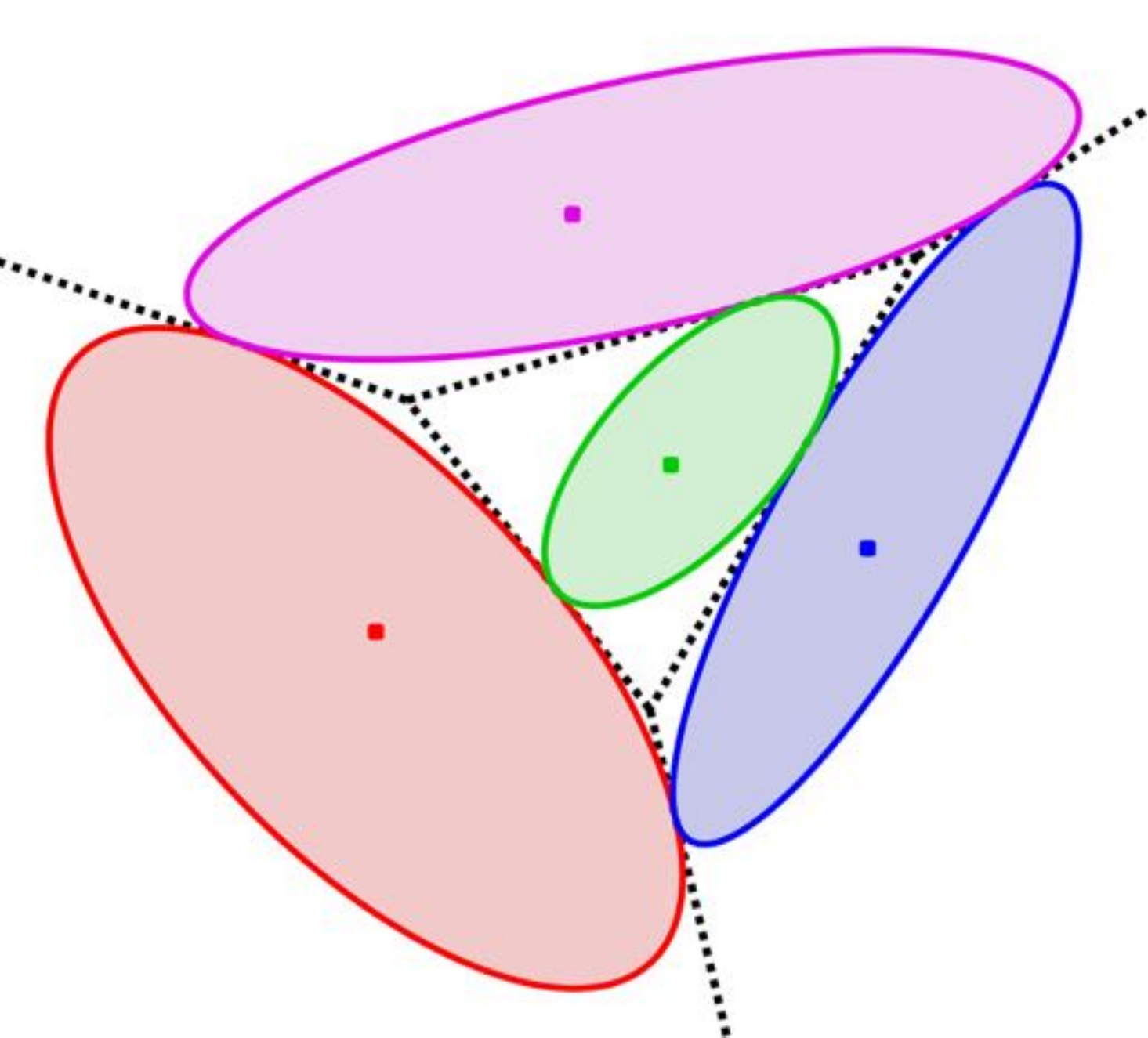}\hspace{10pt} \quad \qquad %
 \includegraphics[height=180pt]{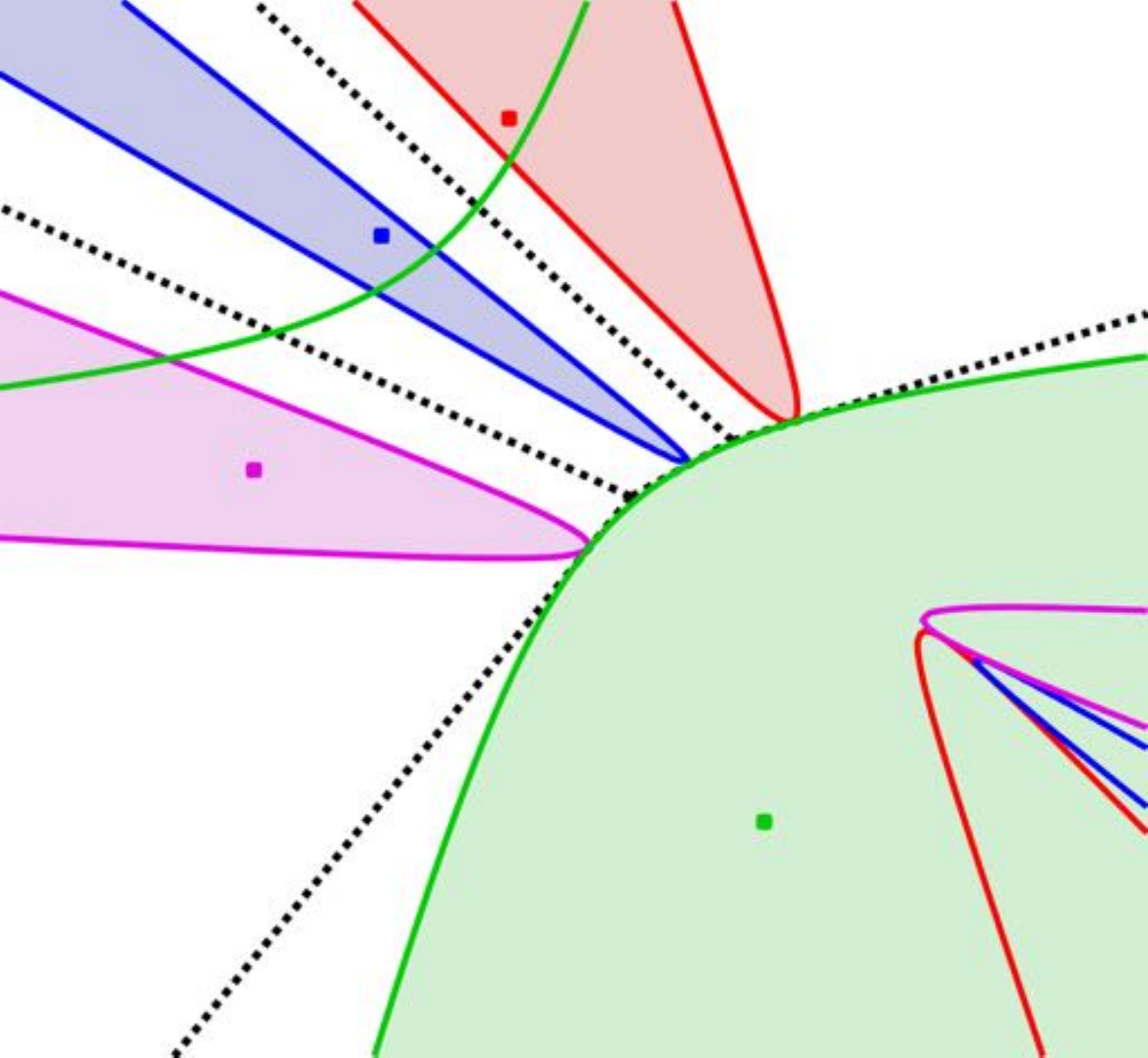}
 \caption{The variety of two  quadratic equations in $\RR^2$ consists of four points.
  The SDP-exact region for the ED~problem consists of conics that are inscribed in the Voronoi cells.
  The conics can be ellipses (left) or hyperbolas (right) depending on the point configuration.
  \label{fig:voronoi2d}}
\end{figure}

\begin{example}[ED~problem for $m=n=2$] \label{ex:voronoi2d} \rm
The variety $V_{\bf f}$ consists of four points in~$\RR^2$.
We seek the point in $V_{\bf f}$ that is closest to a given point $u = (u_1,u_2)$.
The {\em Voronoi decomposition} of $\RR^2$ characterizes the solution.
The SDP-exact region $\mathcal{R}^{ed}_{\bf f}$ consists of four disjoint convex sets, one for each point in~$V_{\bf f}$.
The convex sets are bounded by conics, and are contained in the Voronoi cells of the points.
\Cref{fig:voronoi2d} illustrates $\mathcal{R}^{ed}_{\bf f}$ for two configuration of points in~$\RR^2$: the cells on the left are bounded by ellipses, and on the right by hyperbolas.
Note that in both cases the conics touch pairwise at the bisector lines (cf.~\Cref{thm:finitevariety}).
\end{example}

Our second example is the Max-Cut Problem from discrete optimization.
The SDP relaxation of this problem has been the subject of several papers; see, e.g.,~\cite{Goemans1995,Laurent1995,Helmberg2000}.

\begin{example}[Max-Cut Problem] \rm \label{ex:maxcut}
Let $m=n$ and $f_i(x) = x_i^2-1$, so
$V_{\bf f} = \{-1,+1\}^n$ is the vertex set of the $n$-cube.
We seek a maximal cut in the complete graph $K_n$
where the edge $\{i,j\}$ has weight $c_{ij}$.
In \eqref{eq:qp1} we take $g(x) = \sum_{i,j} c_{ij} x_i x_j$
where $C= (c_{ij})$ is a symmetric~$n {\times} n$ matrix with
$c_{11} = \cdots = c_{nn} = 0$.
Note that these objective functions live in a subspace of dimension ${\binom{n}{2}}$
in $\RR[x]_{\leq 2}$. The dual solution in the SDP relaxation is the {\em  Laplacian}
$$
    \mathcal{L}(C) \,\,= \,\,
\text{ \begin{small}\(
    \left(\begin{matrix}  \,
      -\textstyle\sum_{j\neq 1} c_{1j} \! & c_{12} & c_{13} & \cdots & c_{1n}\\
      c_{12} & \! -\textstyle\sum_{j\neq 2} c_{2j} & c_{23} & \cdots & c_{2n}\\
      c_{13} & c_{23} & \! -\textstyle\sum_{j\neq 3} c_{2j} & \cdots & c_{3n}\\
      \vdots & \vdots & \vdots & \ddots & \vdots\\
      c_{1n} & c_{2n} & c_{3n} & \ldots & \! -\textstyle\sum_{j\neq n} c_{jn} \,\,
    \end{matrix} \right).
\)\end{small} }
$$
The SDP-exact region $\mathcal{R}_{\bf f}$ consists of $2^{n-1}$ spectrahedral cones in $\RR^{\binom{n}{2}}$,
each isomorphic to the set of matrices $C = (c_{ij})$ such that $\mathcal{L}(C)$ is positive semidefinite.
The boundary of this spectrahedron is given by a polynomial of degree $n{-}1$,
namely the determinant of any $(n{-}1) \times (n{-}1)$  principal minor of $\mathcal{L}(C)$.
By the {\em Matrix Tree Theorem}, the expansion of this determinant is the sum
of $n^{n-2}$ monomials in the $c_{ij}$,
one for each spanning tree of $K_n$.
Hence the algebraic boundary of $\mathcal{R}_{\bf f}$
is a (reducible) hypersurface of degree $(n{-}1)2^{n-1}$.

The Max-Cut Problem for $n{=}3$ asks to minimize the inner product with
 $C = (c_{12},c_{13},c_{23})$ over
$\mathcal{T} = \bigl\{(1,1,1),  \,(1,-1,-1), \,(-1,1,-1), \,(-1,-1,1) \bigr\}$.
The feasible region of the SDP relaxation is the {\em elliptope} on the left in \Cref{fig:somosa}.
It strictly contains the tetrahedron ${\rm conv}(\mathcal{T})$.
The region $\mathcal{R}_{\bf f}$
is the set of directions $C$ whose minimum over the elliptope is attained in $\mathcal{T}$.
It consists of the four circular cones over the  facets of the dual of the elliptope.
That dual body is shown in green in \Cref{fig:somosa}, next to the yellow elliptope.
Thus $\mathcal{R}_{\bf f}$
corresponds to the union of the four circular facets of the dual elliptope.
These four circles touch pairwise, just like the four ellipses in \Cref{fig:voronoi2d}.
 The algebraic boundary of $\mathcal{R}_{\bf f}$ has degree $8 = (3-1) 2^{3-1}$.
 \end{example}

The present paper is a sequel to \cite{CAPT}, where the
SDP-exact region for the ED~problem was shown to be full-dimensional in $\RR^n$.
We undertake a detailed study of  $\mathcal{R}_{\bf f}$
and its topological boundary $\partial \mathcal{R}_{\bf f}$. We define the
{\em algebraic boundary} $\partial_{\rm alg} \mathcal{R}_{\bf f}$ to be the Zariski closure
of $\partial \mathcal{R}_{\bf f}$. Our aim is to find
the polynomial defining this hypersurface, or at least to find its degree.
This degree is an intrinsic measure for the geometric complexity of the SDP-exact region.

The material that follows is organized into five sections.
In \Cref{s:2} we introduce the
{\em rank-one region} of a general semidefinite programming problem.
Building on the theory developed in \cite{NRS},  we compute the
degree of the algebraic boundary of this semialgebraic set.

In \Cref{s:3} we turn to the quadratic program \eqref{eq:qp1}.
We introduce its SDP relaxation, and show that $\mathcal{R}_{\bf f}$
coincides with the rank-one region of that relaxation.
In \Cref{thm:betaqp} we determine the degree of $\partial_{\rm alg} \mathcal{R}_{\bf f}$
under the assumption that $f_1,\ldots,f_m$ are generic.
That degree is strictly smaller than the corresponding degree for SDP, which appears in \Cref{thm:beta}.

\Cref{s:4} concerns the Euclidean distance problem and the case when the cost  function $g$ is linear.
\Cref{thm:unionspectrahedron} represents their SDP-exact regions in $\RR^n$
as bundles of spectrahedral shadows. Each shadow lies
in the normal space at a point on $V_{\bf f}$, and is the
linear image of a {\em master spectrahedron} that depends only on ${\bf f}$.
For linear $g$, the region $\mathcal{R}^{\it lin}_{\bf f}$
is determined by the {\em theta body} of Gouveia et al.~\cite{GPT};
see \Cref{thm:thetabody}.
For the ED~problem, $\mathcal{R}^{\it ed}_{\bf f}$ is a tubular neighborhood of
the variety $V_{\bf f}$.
\Cref{fig:voronoi2d}  showed this when
$V_{\bf f}$ consists of four points in $\RR^2$. Analogs  in $\RR^3$ are depicted in Figures \ref{fig:somosas},
 \ref{fig:tacos}, \ref{fig:dualsomosas} (for points) and
Figures \ref{fig:twistedcubic}, \ref{fig:hyperboloids} (for curves).

In \Cref{s:5} we study the algebraic geometry of the SPD-exact region
of the ED~problem.
 \Cref{thm:betaED} gives the degree of the algebraic boundary
$\,\partial_{\rm alg} \mathcal{R}^{\it ed}_{\bf f}$
when $V_{\bf f}$ is a generic complete intersection. It rests on representing our bundle as
a Segre product and projecting it into the ambient space of $V_{\bf f}$.
The abelian surface in \Cref{ex:twoellipticcurves} serves as a nice illustration.

\Cref{s:6} addresses the ED~problem when ${\bf f}$ is not a complete intersection.
\Cref{alg:EDp1} shows how to compute the SDP-exact region.
Several examples demonstrate what can happen.
The dual elliptope on the right of \Cref{fig:somosa}
reappears in five copies in \Cref{fig:dualsomosas}.

\section{The Rank-One Region in Semidefinite Programming}\label{s:2}

Consider a family of semidefinite programming problems with varying cost function:
\begin{equation}
\label{eq:sdp1}
 \min_{X\in\SymRR^{d}} \bigC\bullet X \quad \,\hbox{subject to}
 \,\,\bigA_i  \bullet X = b_i \,\,\,\hbox{for} \,\, i =1,2,\ldots,l,
 \,\,\hbox{and} \,\, X\succeq 0.
\end{equation}
Here $\,\bigC \bullet X = \trace(\bigC X )\,$ is the usual inner product on the space
$\SymRR^{d} \simeq \RR^{\binom{d+1}{2}}$ of symmetric $d \times d$ matrices.
The numbers $b_1,\ldots,b_l \in \RR$ and the matrices  $\bigA_1,\ldots,\bigA_l \in \SymRR^{d}$
are fixed in \eqref{eq:sdp1}, whereas the cost matrix $\bigC$ varies freely over $\SymRR^d$.
The rank-one region $\mathcal{R}_{\bigA,b}$ is a semialgebraic subset of $\SymRR^d$ that depends on $\bigA = (\bigA_1,\ldots,\bigA_l)$ and $b= (b_1,\ldots,b_l)$.
It consists of all matrices $\bigC$ such that~\eqref{eq:sdp1} has a rank-one solution and strict complementarity holds.
See Definition~\ref{def:rank1sol} below.
In this section we study the rank-one region $\mathcal{R}_{\bigA,b}$ and its boundary.
The methods introduced here will be later used in \Cref{s:3} to study the SDP-exact region $\mathcal{R}_{\bf f}$.

The feasible set of~\eqref{eq:sdp1} is the spectrahedron
$ \Sigma_{\bigA,b} =
\bigl\{X {\in} \SymRR^{d} : X{\succeq} 0,\, \bigA_i {\bullet} X {=} b_i \text{ for }1{\leq} i{\leq} l  \bigr\} $.
We assume that $\Sigma_{\bigA,b}$ is non-empty and does not contain the zero matrix.
Then the  region $\mathcal{R}_{\bigA, b}$ is the union
of all normal cones at extreme points of rank one in the boundary of $\Sigma_{\bigA,b}$.

\begin{figure}[htb]
  \centering
  \includegraphics[width=150pt]{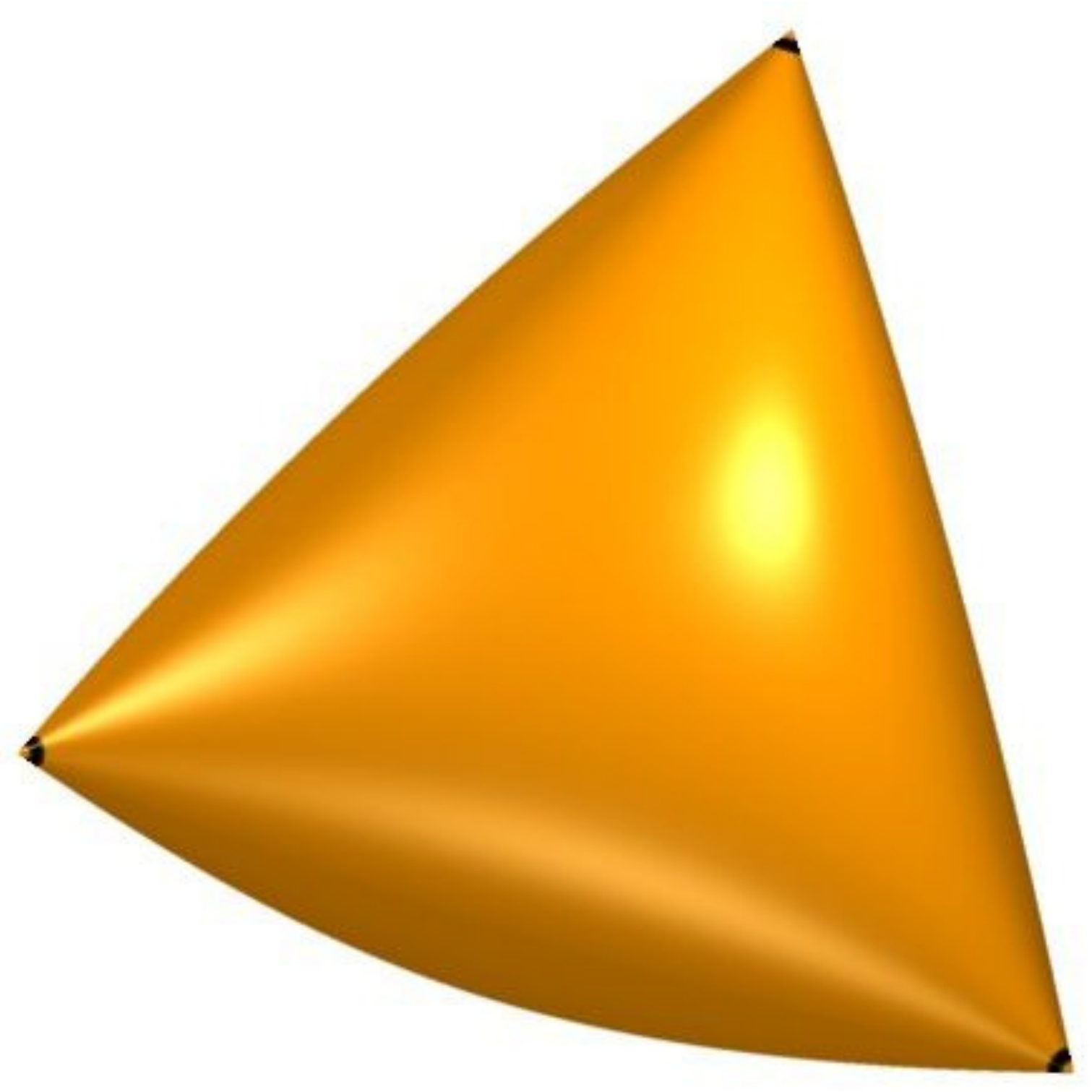}\hspace{10pt} \qquad \qquad %
  \includegraphics[width=150pt]{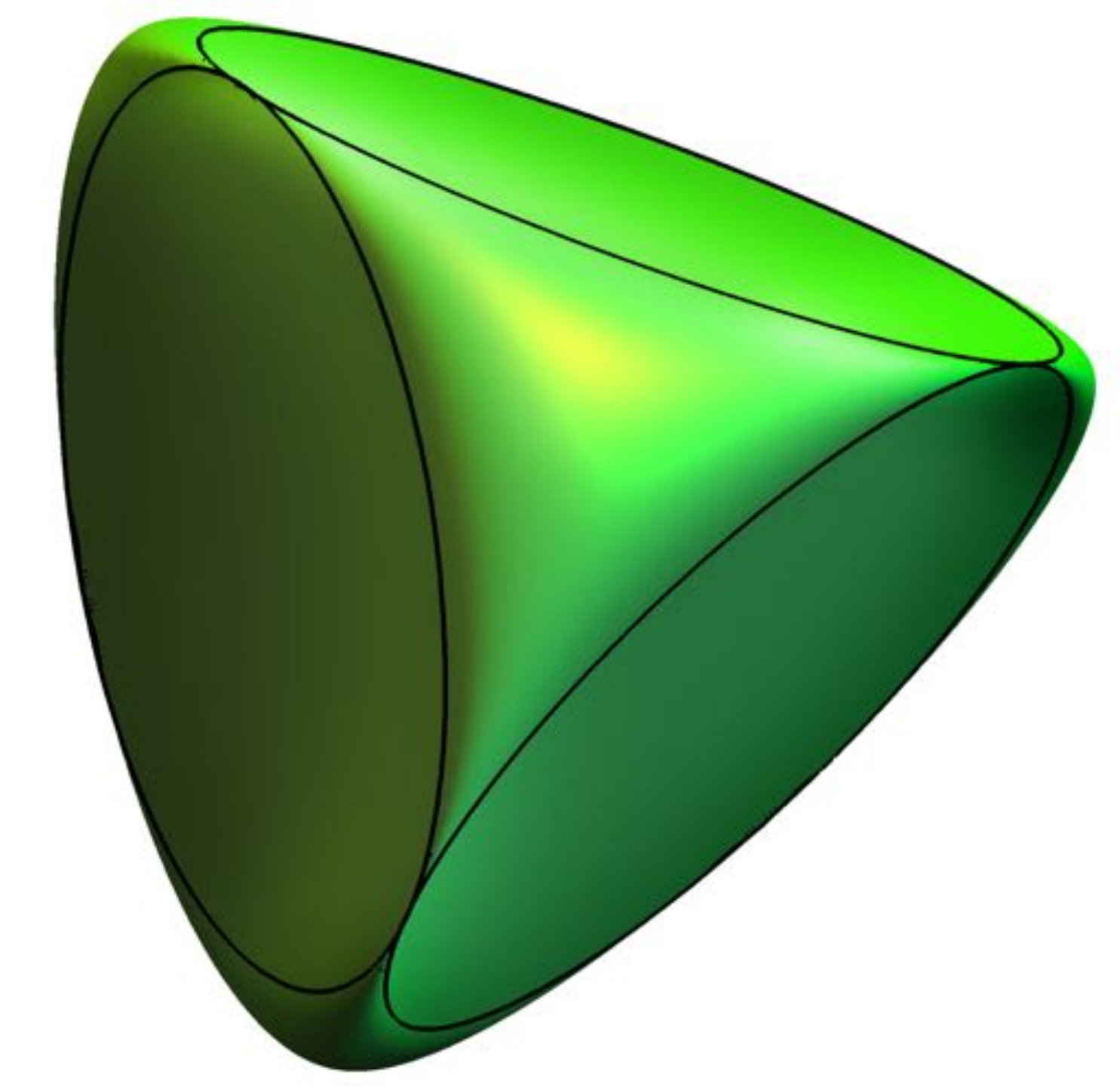}
  \caption{
    The elliptope (left) has four vertices, corresponding to the rank-one matrices.
    The rank-one region consists of the linear forms for which the minimum is attained at a vertex.
    It is given by the cones over the four circular facets of the dual convex body (right).
    \label{fig:somosa}}
\end{figure}

\begin{example}[$d=l=3$] \rm
\label{ex:somosarevisted}
The convex bodies in \Cref{fig:somosa} arise for Max-Cut with $n=3$ in \Cref{ex:maxcut}.
The spectrahedron $\Sigma_{\bigA,b}$ on the left is the elliptope.
It is bounded by {\em Cayley's cubic surface}.
The four nodes are the rank-one points in $\partial \Sigma_{\bigA,b}$.
The dual convex body, shown on the right, is bounded by the
{\em quartic Steiner surface} and it has four circular facets.
The rank-one region $\mathcal{R}_{\bigA,b}$ is given
by the interiors of these four circles, viewed as cones in~$\mathcal{S}^3$.
\end{example}

The semidefinite program that is dual to~\eqref{eq:sdp1} has the form:
\begin{equation}
 \label{eq:sdp2}
  \max_{Y\in\SymRR^{d},\,\lambda\in\RR^{l}} \quad  b^T \lambda
  \qquad\text{subject to}\qquad
  Y \,=\, \bigC - \textstyle\sum_{i=1}^l \lambda_i \bigA_i  \quad {\rm and} \quad     Y \succeq 0.
 \end{equation}
The following {\em critical equations} express
the complementary slackness condition that links the
optimal solution $X \succeq 0$ of the primal~\eqref{eq:sdp1}  and the
optimal solution $Y \succeq 0$ of the dual~\eqref{eq:sdp2}:
\begin{equation} \label{eq:criticaleqns}
 \bigA_i \bullet X = b_i \;\text{ for } 1 \leq i\leq l \quad {\rm and} \quad
    Y = \bigC - \textstyle\sum_{i=1}^l \lambda_i \bigA_i  \quad {\rm and} \quad
    X \cdot Y = 0.
\end{equation}
Recall that \emph{strict complementarity} holds if $\rank(X) \!+\! \rank(Y) \nobreak\!=\! d$.
The rank-one region is the semialgebraic set given by the critical equations and strict complementarity, as follows:

\begin{definition} \label{def:rank1sol}
  The \emph{rank-one region} $\mathcal{R}_{\bigA,b}$ is the set of all $\,\bigC
  \in \SymRR^{d}$ for which there exist
  $\lambda \in \RR^l$ and $X,Y\in \SymRR^{d}$ such that
  $\,X , Y\succeq 0$, $\,  \rank(X)= 1 $,
  $\, \rank(Y)= d-1\,$ and  \eqref{eq:criticaleqns} holds.
\end{definition}

\begin{remark} \rm
  The above construction can be extended to define the \emph{rank-$r$ region} for other values of~$r$.
  It is an interesting open problem to investigate the geometry of these regions.
\end{remark}

The results that follow hold for \emph{generic} instances of the matrices $\bigA_i$ and the vector $b$.
This implies that the results hold for ``almost all'' instances of $(\bigA, b)$,
i.e., outside a set of Lebesgue measure zero.
More precisely, a property holds generically if there is a polynomial  $f$
in the entries of $\bigA$ and $b$ such that it
holds whenever this polynomial does not vanish.

Genericity was also a standing assumption in the derivation of the {\em algebraic degree of semidefinite programming} by Nie~et al.~\cite[\S2]{NRS}.
That degree, denoted $\delta(l,d,r)$,
 is the number of complex solutions $(X,Y)$ of the critical equations (\ref{eq:criticaleqns}) for the SDP \eqref{eq:sdp1},
with $l$ constraints for $d \times d$ matrices, assuming that
  ${\rm rank}(X) = d-r$ and ${\rm rank}(Y) = r$.
 A formula for general $r$ was given in \cite{bothmer}.
 The easier case $r=d-1$ appeared in \cite[Theorem~11]{NRS}:

\begin{proposition} \label{prop:delta}
The algebraic degree of rank-one solutions $X$ to the SDP  in \eqref{eq:sdp1} equals
\begin{equation*}
 \delta(l,d,d-1)\,\, = \,\,2^{l-1} \binom{d}{l}.
 \end{equation*}
\end{proposition}

The following geometric formulation of SDP was proposed in~\cite[eqn.~(4.1)]{NRS}.
Let $\mathcal{V}$ be the $(l-1)$-dimensional subspace of
$\SymRR^d$ spanned by $\{\bigA_2,\ldots,\bigA_l\}$, and let $\mathcal{U}$
be the $(l+1)$-dimensional subspace of $\SymRR^d$  spanned by
$\{\bigC,\bigA_1\}$ and $\mathcal{V}$. This specifies a dual pair of flags
\begin{equation}
\label{eq:flags}
 \mathcal{V} \, \subset \,\mathcal{U} \, \subset \, \SymRR^d \qquad {\rm and} \qquad
\mathcal{U}^\perp \, \subset \,\mathcal{V}^\perp \, \subset \, \SymRR^d . \end{equation}
See \cite[eqn.~(3.3)]{NRS}. The critical equations \eqref{eq:criticaleqns} can now be written as
\begin{equation}
\label{eq:UV}  X \,\in \,\mathcal{V}^\perp \quad {\rm and} \quad
Y \, \in \,\mathcal{U} \quad {\rm and} \quad X \cdot Y = 0 .
\end{equation}
The SDP problem \eqref{eq:sdp1} is equivalent to solving \eqref{eq:UV} subject  to $X , Y \succeq 0$.
The algebraic degree $\delta(l,d,r)$ is the number of complex solutions to
\eqref{eq:UV} with ${\rm rank}(X) = d-r$ and ${\rm rank}(Y) = r$.
The dual pair of flags in~\eqref{eq:flags} will also play a critical role in our derivation of the degree of the boundary of the rank-one region.

\begin{remark} \label{rem:theseequiv} \rm
If the matrices $\bigA_i$ and the scalars $b_i$ are generic
then strict complementarity always holds~\cite[Corollary 8]{NRS},
and hence the following conditions are equivalent:
\begin{itemize}
\item The primal SDP problem \eqref{eq:sdp1} has a unique optimal matrix $X$ of rank $1$. \vspace{-0.1in}
\item The dual SDP problem \eqref{eq:sdp2} has an optimal matrix $Y$ of rank $d-1$. \vspace{-0.1in}
\item The system \eqref{eq:UV} has a solution $(X,Y)$ with $\rank(X) = 1$
and $X,Y \succeq 0$.
\end{itemize}
These conditions characterize the set of cost matrices $\,\bigC$ that lie in the rank-one region $\mathcal{R}_{\bigA,b}$.
\end{remark}

Suppose that the rank-one region $\mathcal{R}_{\bigA,b}$ is non-empty.
The topological
 boundary $\partial \mathcal{R}_{\bigA,b} $ is a closed
 semialgebraic set of pure codimension one in $\SymRR^{d}$. Its Zariski closure
 $\partial_{\rm alg} \mathcal{R}_{\bigA,b} $ is an algebraic  hypersurface, called the {\em rank-one boundary}.
We view this hypersurface either in the complex affine space $\CC^{\binom{d+1}{2}}$,
or in the corresponding projective space
$\PP( \SymRR^d) \simeq \PP^{\binom{d+1}{2}-1}$.
 By construction, the polynomial
defining $\partial_{\rm alg} \mathcal{R}_{\bigA,b}$ has coefficients in the field
generated by the entries of $\bigA$ and $b$ over $\QQ$.
The {\em rank-one boundary degree} is the degree of this polynomial:
$$ \beta(l,d) \,\, = \,\, \deg\bigl( \partial_{\rm alg}\mathcal{R}_{\bigA,b} \bigr). $$
Our main result in this section furnishes a formula for the degree of the rank-one boundary.

\begin{theorem} \label{thm:beta}
Let $ 3 \leq l \leq d$ and consider the SDP with generic $\bigA$ and $b$, as given in \eqref{eq:sdp1}.
The degree of the hypersurface $\partial_{\rm alg} \mathcal{R}_{\bigA,b}$
that bounds the rank-one region $\mathcal{R}_{\bigA,b}$ equals
\begin{equation}
\label{eq:beta} \beta(l,d) \,\, = \,\, 2^{l-1} (d-1)\binom{d}{l} - 2^l\binom{d}{l+1} .
\end{equation}
\end{theorem}

\begin{table}[h]
\begin{center} \qquad
  \setlength{\tabcolsep}{7pt}
\begin{tabular}{c|*{5}{c}}
\toprule
  \multicolumn{6}{c}{Algebraic degrees $\delta(l,d,d-1)$}\\
\midrule
  $l\backslash d$ & 3 & 4 & 5 & 6 & 7  \\
\midrule
2 &  6 & 12 & 20 & 30 & 42 \\
3 & 4 & 16 & 40 & 80 & 140 \\
4 &   & 8 & 40 & 120 & 280 \\
5 &   &   & 16 & 96 & 336 \\
6 &   &   &   & 32 & 224 \\
7 &   &   &   &   & 64 \\
\bottomrule
 \end{tabular} \qquad \qquad
  \setlength{\tabcolsep}{9pt}
\begin{tabular}{c|*{5}{c}}
\toprule
  \multicolumn{6}{c}{Rank-one boundary degrees $\beta(l,d)$}\\
  \midrule
  $l\backslash d$ & 3 & 4 & 5 & 6 & 7  \\
\midrule
2 & 4 & 10 & 20 & 35  & 66 \\
3 & 8 & 40 & 120 & 280 & 560 \\
4 &   & 24 & 144 & 504 & 1344 \\
5 &   &   & 64 & 448 & 1792 \\
6 &   &   &   & 160 & 1280 \\
7 &   &   &   &   & 384 \\
\bottomrule
 \end{tabular}
\vspace{-0.11in}
\end{center}
\caption{ Algebraic degrees and boundary degrees of SDP.
 \label{tab:degrees} }
  \end{table}

\Cref{tab:degrees} illustrates \Cref{prop:delta}
and \Cref{thm:beta}. It shows the algebraic degrees of rank-one SDP
on the left, and  corresponding  rank-one boundary degrees on the right.
The entry for $l=d=3$ equals $8 = 2{+}2{+}2{+}2$, as argued in \Cref{ex:somosarevisted}
and seen in \Cref{fig:somosa}.
The first row ($l=2$)  is not covered by \Cref{thm:beta}.
This case requires special consideration.

\begin{proposition} \label{prop:m2}
  If $\,l=2$ then the rank-one region $\mathcal{R}_{\bigA,b}$ is dense in the matrix space $\SymRR^d$.
If $\bigA, b$ are generic then
$\partial\mathcal{R}_{\bigA,b} = \SymRR^d \backslash \mathcal{R}_{\bigA,b}$
is a hypersurface of degree $\beta(2,d) = \binom{d+1}{3}$.
\label{prop:betameq2}
\end{proposition}

\begin{proof}
The semialgebraic set  $\mathcal{R}_{\bigA,b}$ is dense in
the classical topology on $\SymRR^d$ because the {\em Pataki range} \cite[\S 3]{NRS}
consists of a single rank for $l=2$. This means that, for almost all
cost matrices $\bigC$, there is an optimal pair $(X,Y)$ that satisfies ${\rm rank}(X)=1$ and
${\rm rank}(Y) = d-1$. The boundary $\partial\mathcal{R}_{\bigA,b}$ is the set of
$\bigC$ such that the optimal matrix $Y = \bigC - \lambda_1 \bigA_1 - \lambda_2 \bigA_2$
has rank $\leq d-2$. The polynomial in $\bigA_1,\bigA_2,\bigC$ that
defines this hypersurface is
the {\em Chow form} of the determinantal variety $\{{\rm rank}(Y) \leq d-2 \}$.
This variety has codimension three in $\PP(\SymRR^d)$ and degree $\binom{d+1}{3}$ (see \cite[Prop. 12(b)]{HarrisTu}).
This is the degree of the Chow form in the entries of $\bigC$, and hence it is the degree
of our hypersurface $\partial_{\rm alg} \mathcal{R}_{\mathcal{A}, b}$.
\end{proof}

The proof of  Theorem \ref{thm:beta} requires additional concepts from algebraic geometry.
We work with the {\em Veronese variety}  $\,\PP^{d-1} \hookrightarrow \PP(\SymRR^d)$.
By \cite[Proposition~12]{NRS}, its {\em conormal variety}~is
 \begin{equation} \label{eq:CV}
\!\! CV \,= \, \bigl\{ (X,Y) \in \PP(\SymRR^d) \times  \PP(\SymRR^d) \,:\,
 XY = 0 \,\, {\rm and} \,\, {\rm rank}(X) = 1 \,\, {\rm and} \,\,
 {\rm rank}(Y) \leq d-1 \bigr\}.  \,\,
 \end{equation}
As in \cite[Theorem~10]{NRS}, we consider the corresponding class
  $\,[CV]\,$ in the cohomology ring
\begin{equation}
\label{eq:cohomology} H^*\bigl(\,\PP(\SymRR^d) \times  \PP(\SymRR^d),\, \ZZ \,\bigr) \, = \,
\ZZ[\,s,\,t\, ] \, / \bigl\langle \,s^{\binom{d+1}{2}}, \,t^{\binom{d+1}{2}} \,\bigr\rangle.
\end{equation}
Its coefficients are the {\em polar degrees} of the Veronese variety. By \Cref{prop:delta}, we have
\begin{equation}
[CV] \,\,\, = \,\, \, \sum_{l=1}^d 2^{l-1} \binom{d}{l} \cdot s^{\binom{d+1}{2}-l} t^l.
\label{eq:bidegreeCV}
\end{equation}
We represent $CV$ by its pullback under the Veronese map $x \mapsto X = x x^T$
on the first factor.  Thus the conormal variety equals
$\,CV \,=\, \bigl\{\, (x,Y)  \, : \, Y x=0 \,,\,
\, {\rm det}(Y) = 0 \, \bigr\}$ in $\PP^{d-1} \times
\PP(\SymRR^d) $.

We note that the following {\em boundary variety} is irreducible of codimension one in $CV$:
\begin{equation} \label{eq:BV}
\begin{aligned}
 BV  & \,=\,  \bigl\{ (X,Y) \in \PP(\SymRR^d) \times  \PP(\SymRR^d) \,:\,
 XY = 0 \, , \,\, {\rm rank}(X) = 1 \,\, {\rm and} \,\,
 {\rm rank}(Y) \leq d-2 \bigr\} \\
 & \,\simeq\,   \bigl\{\, (x,Y) \in \PP^{d-1} \times  \PP(\SymRR^d) \,\,:\,\,
 Y x \,=\, 0 \quad {\rm and} \quad {\rm rank}(Y) \leq d-2 \,\bigr\}.
\end{aligned}
\end{equation}
By the last item in \Cref{rem:theseequiv}, the algebraic boundary of $\mathcal{R}_{\bigA,b}$ is contained in~$BV$.

Let $Y = (y_{ij})$ be a symmetric $d \times d$ matrix
and $x = (x_1\;x_2\,\cdots\,x_d)^T$ a column vector.
Their entries are the variables of the polynomial ring $T = \CC[x_1,\dots,x_d,y_{11}, y_{12}, \dots,y_{dd}]$.
Subvarieties of $ \PP^{d-1} \times  \PP(\SymRR^d)$ are defined by bihomogeneous ideals in $T$.
The ideal of the conormal variety equals $I_{CV} = \langle Y x, \det(Y) \rangle$.
The ideal of the boundary variety equals  $I_{BV} = I_{CV} + \operatorname{Min}_{d-1}(Y)$.
The latter is the ideal generated by the $(d{-}1) {\times} (d{-}1)$ minors of~$Y$.

\begin{proof}[Proof of \Cref{thm:beta}]
  Let $C = (c_{ij})$ denote the {\em adjugate} of $Y$. The entry
$c_{ij}$ of this $d \times d$ matrix is the $(d{-}1) \times (d{-}1)$ minor of $Y$ complementary to $y_{ij}$.
We are interested in the divisor in the smooth variety $CV$ that is defined by the equation $c_{11} = 0$.
We claim that this divisor is the sum of  the boundary divisor $BV$ and the divisor defined by $x_1^2 = 0$.

To prove this claim, we consider the ideals
  $I := I_{CV} + \langle c_{11}\rangle$ and
  $J := I_{CV} + \operatorname{Min}_{d-1}(Y) \cdot \langle x_1^2 \rangle$   in $T$.
  It suffices to show $I = \operatorname{sat}(J)$, the saturation with
  respect to $\langle x_1,\dots,x_d\rangle$.
  Consider the $d \times (d+1)$ matrix $ ( x \;|\; C )$.
  The ideal $M :=\operatorname{Min}_2(x \;|\; C )$
  is contained in $I_{CV}$. Combining two of its generators, we find $c_{ij}x_1^2 - c_{11}x_ix_j \in M$.
  Therefore the generator $c_{ij}x_1^2$ of $J$ lies in $M + \langle\, c_{11} \,\rangle \subset I$.
  So $J\subseteq I$, and since $I$ is saturated, $\operatorname{sat}(J)\subseteq I$.
  For the reverse inclusion we need to show that $c_{11} \in \operatorname{sat}(J)$.
  This follows by noting that $ c_{11}x_k^2 - c_{kk}x_1^2 \in M$,
  and thus $c_{11} x_k^2 \in M + \operatorname{Min}_{d-1}(Y) \cdot \langle x_1^2\rangle \subset J$.
  Therefore, $I = \operatorname{sat}(J)$ and the claim follows.

We now compute the class of $BV$ in the cohomology ring (\ref{eq:cohomology}).
The minor $c_{11}$ defines a hypersurface of degree $d-1$ in $\PP(\SymRR^d)$, so its class
 is $(d-1) t$. The class of $\{x_1^2=0\}$ is twice the hyperplane class
in $\PP^{d-1}$.
It is the pullback of $[\{x_{11}=0\}]=s$ under the Veronese map into $\PP(\SymRR^d)$.
Here $x_{11}$ is the upper left entry in the matrix $X = x x^T$.
We multiply these classes with $[CV]$ as in~\eqref{eq:bidegreeCV},
and thereafter we subtract.
By the claim we proved, this gives
\begin{align*}
  [BV] &\;=\; [\,CV \,\cap\, \{ c_{11} = 0 \}]\, -\, [\,CV \cap \{x_1^2=0\}] \\
 & \;=\; \bigl(\,(d-1)t \,-\,s \,\bigr) \cdot [CV]
 \;=\; \sum_{l=2}^d   \beta(l,d) \cdot s^{\binom{d+1}{2}-l} t^{l+1},
\end{align*}
where  the coefficients of the resulting binary form are the expressions on
the right of \eqref{eq:beta}.

The following argument shows that the class $[BV]$ encodes the rank-one boundary degrees.
Suppose the cost matrix $\bigC$ travels
on a generic line in $\SymRR^d$
from the inside to the outside of the rank-one  region $\mathcal{R}_{\bigA,b}$.
For almost all points $\bigC$
on that line,  the optimal pair  $(X,Y)$ is unique. Before $\bigC$
crosses the boundary $\partial \mathcal{R}_{\bigA,b}$, the optimal pair satisfies ${\rm rank}(X) = 1$ and
${\rm rank}(Y) = d-1$. Immediately after $\bigC$ crosses
$\partial \mathcal{R}_{\bigA,b}$, we have
${\rm rank}(X) =2$ and ${\rm rank}(Y) = d-2$. At
the transition point, the optimal pair $(X,Y)$ lies in the variety $BV$.

Consider the intersection of $BV$ with the product of the
codimension-$(l-1)$ plane $\PP(\mathcal{V}^\perp)$
and the subspace $\PP(\mathcal{U}') \simeq \PP^{l+1}$
 spanned by $\bigA_1,\ldots,\bigA_l$
and the line on which $\bigC$ travels.
The points in that intersection
are the pairs $(X,Y) \in BV$ that arise as $\bigC$ travels along the line.
The number of such complex intersection points is the coefficient of
$s^{\binom{d+1}{2}-l} t^{l+1}$ in $[BV] $.

We need to argue that
the inclusion (\ref{eq:flags}) poses no restriction
on the products of subspaces we intersect with, i.e.,~for
generic flags $\mathcal{V} \subset \mathcal{U'}$
with ${\rm dim}(\mathcal{U'}/ \mathcal{V}) = 3$, all intersections
with $BV$ are transverse and reduced.
To this end, let $X_0$ be the rank-one $d \times d$ matrix with a single one in the first entry, and let $Y_0$ be the diagonal $d \times d$ matrix with two zeros followed by $d-2$ ones.
Then an affine neighborhood of
$(X_0,Y_0) $ in $ \mathbb{P}(\mathcal{S}^d) \times \mathbb{P}(\mathcal{S}^d)$
can be given as the direct sum of the spaces parametrized by
\[
\begin{pmatrix}
1 & x_{1 2} & \cdots & x_{1,d} \\
x_{1 2} & x_{2 2} & \cdots & x_{1,d} \\
\vdots & \vdots & \ddots & \vdots \\
x_{1,d} & x_{2,d} & \cdots & x_{d,d}
\end{pmatrix} \quad \text{ and } \quad
\begin{pmatrix}
y_{1 1} & y_{1 2} & y_{1 3} & \cdots & y_{1 , d-1} & y_{1,d} \\
y_{1 2} & y_{2 2} & y_{2 3} & \cdots &  y_{2 , d-1} & y_{2,d} \\
y_{1 3} & y_{2 3} & 1+y_{3 3} & \cdots & y_{3 , d-1} & y_{3,d} \\
\vdots & \vdots & \vdots & \ddots & \vdots & \vdots \\
y_{1, d-1} & y_{2, d-1} & y_{3, d-1} & \cdots & 1+y_{d-1 , d-1} & y_{d-1 , d} \\
y_{1,d} & y_{2,d} & y_{3,d} & \cdots & y_{d-1 , d}& 1
\end{pmatrix}.
\]
The linear terms in the coordinates of the matrix equation $XY=0$ are
\begin{equation}
\label{eqn:tangentSpaceGens}
y_{1 1}, \,y_{1 2},\, x_{1 3} + y_{1 3},\, \ldots,\,
x_{1,d} + y_{1,d}\ \text{ and } \ x_{2 3},\, \ldots, \,x_{d,d},
\end{equation}
for a total of
% $d + \binom{d}{2} - 1 =$
$\binom{d+1}{2}-1$ forms.
To show that the intersection described above is transverse for generic flags
$\mathcal{V} \subset \mathcal{U'}$, it suffices to find one instance for which
$BV \cap (\mathbb{P}(\mathcal{V}^\perp) \times \mathbb{P}(\mathcal{U'})) = \{(X_0,Y_0)\} $
in the neighborhood of
$(X_0,Y_0) \in \mathbb{P}(\mathcal{S}^d) \times \mathbb{P}(\mathcal{S}^d)$
defined above.

Let $\mathbb{P}(\mathcal{V}^\perp)$ be determined by the
vanishing of the $l-1$ forms $x_{1 2},\dots, x_{1\, l-1},x_{2 2}-x_{2 3}$ and $\mathbb{P}(\mathcal{U'})$ by the $\binom{d+1}{2} - l$ forms
$y_{1,l},\dots,y_{1,d},y_{2 2}+y_{2 3},y_{2 3}, \dots, y_{d-1, d}$.
Combining these forms with those in (\ref{eqn:tangentSpaceGens}), we get $2(\binom{d+1}{2}-1)$ independent linear forms.  This (highly non-generic) choice yields a transverse intersection. We conclude that the intersection
$BV \cap (\mathbb{P}(\mathcal{V}^\perp) \times \mathbb{P}(\mathcal{U'}))$ is transverse and reduced at $(X_0,Y_0)$ also for generic choices of $\mathcal{V} \subset \mathcal{U'}$.
 \end{proof}

\section{From Semidefinite to Quadratic Optimization}\label{s:3}

We now model the quadratic optimization problem \eqref{eq:qp1}
as a special case of the semidefinite program \eqref{eq:sdp1}.
To this end, we set $l=m+1$, $d=n+1$, and we use indices
that start at $0$ and run to $m$ and $n$ respectively.
Let $\bigA_0$ be the rank-one matrix $E_{00}$
whose entries are $0$ except for the entry $1$ in the upper left corner.
The following two conditions are equivalent:
\begin{equation}
\label{eq:rank1constraint}
  \bigA_0 \bullet X = 1, \,\,\rank(X) = 1 \; \,{\rm and}\; \,X \succeq 0
  \; \iff \;
  X = (1,x_1,\ldots,x_{n})^T  (1,x_1,\ldots,x_{n}).
\end{equation}
Setting $b = (1,0,\ldots,0)$ and imposing the rank constraint in \eqref{eq:rank1constraint},
our SDP in \eqref{eq:sdp1}  is equivalent to
minimizing  a quadratic function in $x$ subject to the constraints
 $\bigA_1 \bullet X = \cdots = \bigA_{m} \bullet X = 0 $.

 To apply SDP to the problem \eqref{eq:qp1}, with $m$ quadratic constraints in $n$
 variables, we set
$$
g({x}) \, \,= \,\, {x}^T {C} {x} + c^T {x}  \quad {\rm and} \quad
  f_i({x}) \,\,= \,\, {x}^T {A}_i {x} + 2 a_i^T {x} + \alpha_i
  \quad\text{ for }1\leq i \leq m.
$$
The matrices  ${C},{A}_i\in\SymRR^{n}$, the vectors $c,a_i\in\RR^{n}$, and the scalars $\alpha_i\in\RR$,
give the entries in
\begin{equation}
\label{eq:vecmat}
  \bigC := \begin{bmatrix} 0 & c^T \\ c & {C} \end{bmatrix}, \;\,
  \bigA_0 := \begin{bmatrix} 1 & 0 \\ 0 & 0 \end{bmatrix}, \;\,
  \bigA_i := \begin{bmatrix} \alpha_i & a_i^T \\ a_i & {A}_i \end{bmatrix} \,\in\, \SymRR^{d}.
\end{equation}
If we now also set
$\,X = \bigl(\begin{smallmatrix}1\\x\end{smallmatrix}\bigr) (1\; x^T) $
then \eqref{eq:qp1} is precisely the SDP \eqref{eq:sdp1}.
In other words, \eqref{eq:qp1} is equivalent to \eqref{eq:sdp1}
with the additional constraint $\,\rank(X) = 1$.
The SDP~\eqref{eq:sdp1} is called the \emph{Shor relaxation} of the quadratic program~\eqref{eq:qp1}.
We say that the relaxation is \emph{exact} if the primal optimal solution $X^*$ of the SDP is unique and has rank one.

The SDP arising as a relaxation of a quadratic program has two distinctive features:
the matrix $\bigA_0$ is the rank-one matrix $E_{00}$,
and we fix the values $b_0 {=} 1$, $c_{00} {=} b_1 {=} \cdots {=} b_m {=} 0$.
The last $m+1$ equations pose no restriction:
they can be achieved by adding multiples of $\bigA_0$ to $\bigC,\bigA_1,\ldots,\bigA_m$.
The only truly special feature of this SDP is that $\bigA_0$ has rank one.

\begin{remark} \label{rem:QP=SDP} \rm
The Shor relaxation of a quadratic optimization problem in $\RR^n$ is a semidefinite program in $\SymRR^{n+1}$ in which one constraint matrix $\bigA_0$ is rank-one.
\end{remark}

We fix the identifications in
(\ref{eq:vecmat}) throughout this section.
In particular, we will define the SDP-exact region as the restriction of the rank-one region to SDP's coming from quadratic programs.
Consider the  {\em Lagrangian} function
\begin{equation} \label{eq:lagrangian}
\mathcal{L}(\lambda,{x}) \,\, := \,\,
g(x) \,-\, \sum_{i=1}^{m} \lambda_i f_i({x}).
\end{equation}
This polynomial
is quadratic in ${x}$.
Its Hessian with respect to ${x}$ is the symmetric $n {\times} n$~matrix
\begin{equation} \label{eq:hessian}
  {\rm H}(\lambda) \,\, := \,\,
  \biggl(  \frac{\partial^2 \mathcal{L} }{ \partial x_i \partial x_j} \biggr)_{\! 1 \leq i,j \leq n}
  = \quad C  - \sum_{i=1}^m {\lambda}_i{A}_i.
\end{equation}
The entries of the matrix ${\rm H}(\lambda)$ are affine-linear in $\lambda = (\lambda_1,\ldots,\lambda_m)$.

The SDP-exact region is obtained by specializing \Cref{def:rank1sol} to the matrices in (\ref{eq:vecmat}):

\begin{definition} \label{def:rank1sol2}
  The \emph{SDP-exact region} $\mathcal{R}_{\bf f}$ is the set of all matrices $\,\mathcal{C}\in \SymRR^{n+1}$ such that
\begin{equation}
\label{eq:lagrangian1}
    {\rm H}(\lambda) \,\succ \, 0 \quad
    {\rm and}  \quad
    c - \sum_{i=1}^m {\lambda}_i a_i + {\rm H}(\lambda) {x}\, =\, 0 \qquad
    \text{for some ${x} \in V_{\bf f}$ and $\,{\lambda}\in\RR^m$.}
\end{equation}
\end{definition}

The condition \eqref{eq:lagrangian1} has a natural interpretation
in the setting of constrained optimization. It says that the
Hessian of the Lagrangian is positive definite at the
optimal solution.

\begin{remark} \label{rem:shadows} \rm
  \Cref{def:rank1sol2} expresses $\mathcal{R}_{\bf f}$
  as a union of spectrahedral shadows \cite{Sch, SS}.
    To see this, fix a point ${x}$ in $V_{\bf f}$. The
constraints \eqref{eq:lagrangian1} define a spectrahedron
${S}_{{x}}$ in the space with coordinates $({\lambda}, {C}, c)$.
The SDP-exact region for $x$ is the image of ${S}_{{x}}$
under the projection onto the coordinates $({C},c)$.
This image is a spectrahedral shadow. \Cref{def:rank1sol2}
says that $\mathcal{R}_{\bf f}$ is the union of these shadows.
We shall return to this point in \Cref{thm:unionspectrahedron}.
\end{remark}

The main result in this section is the extension of \Cref{prop:delta} and \Cref{thm:beta} to
quadratic optimization. Let $N = \binom{n+2}{2}-1$ and consider the map
$\pi : \PP^N \times  \PP^N \dashrightarrow \PP^N \times \PP^{N-1}$
that deletes the upper left entry $y_{00}$ of the matrix $Y$.
Let $CV'  = \overline{\pi(CV)} $ denote the closed image of the conormal
variety $CV$ in \eqref{eq:CV}
under the map $\pi$, and similarly let
$BV' = \overline{\pi(BV)}$ denote the closed image of
the boundary variety in \eqref{eq:BV}.
Algebraically, we compute these projected varieties by
eliminating the unknown $y_{00}$ from the
defining ideals of \eqref{eq:CV} and \eqref{eq:BV}.

\begin{proposition}
The algebraic degree of \eqref{eq:qp1}
is given by $[CV']$ in $H^*(\PP^N{\times}\PP^{N-1})$. We have
\begin{equation}
\label{eq:algdegQP0}
[ CV'] \,\, = \,\, \sum_{m=0}^{n}
2^{m} \binom{n}{m}\cdot  s^{\binom{n+2}{2}-(m+1)} t^{m} .
\end{equation}
Similarly,  the degree of $\partial_{\rm alg} \mathcal{R}_{\bf f}$
is given by the class of the projected boundary variety  $BV'$.
\end{proposition}

\begin{proof}
The map $\pi$ is the projection from the special point $\bigA_0 = E_{00}$ in $\PP^N$.
In the proof of \Cref{thm:beta}, we intersect
$CV$ and $BV$ with products of complementary linear spaces.
The situation is  the same here, except that we now require the linear space
in the second factor to contain the point $ \bigA_0$. Thus, our
counting problem is equivalent to intersecting the projections via $\pi$
by products of generic linear spaces of complementary dimension.
The formula in \eqref{eq:algdegQP0}
is the algebraic degree of quadratic programming, which is found in
\cite[eqn.~(3.1)]{NR}.
\end{proof}

\begin{table}[h]
\begin{center} \qquad
\begin{tabular}{c|*{5}{c}}
\toprule
  \multicolumn{6}{c}{Algebraic degrees of QP}\\
\midrule
  $m\backslash n$ & 2& 3 & 4 & 5 & 6 \\
\midrule
1 & 4 & 6 & 8 & 10 & 12 \\
2 & 4 & 12 & 24 & 40 & 60 \\
3 &   & 8 & 32 & 80 & 160 \\
4 &   &   & 16 & 80 & 240 \\
5 &   &   &   & 32 & 192 \\
6 &   &   &   &   & 64 \\

 \end{tabular} \qquad \qquad
\begin{tabular}{c|*{5}{c}}
\toprule
  \multicolumn{6}{c}{Boundary degrees $\beta_{QP}(m,n)$} \\
\midrule
  $m\backslash n$ & 2& 3 & 4 & 5 & 6 \\
\midrule
 1 & 6 & 12 & 20 & 30 & 42 \\
 2 & 8 & 32 & 80 & 160 & 280 \\
 3 &   & 24 & 120 & 360 & 840 \\
 4 &   &   & 64 &  384 & 1344 \\
 5 &   &   &   & 160 & 1120 \\
 6 &   &   &   &   &  384 \\
 \end{tabular}
\vspace{-0.11in}
\end{center}
\caption{\label{tab:degreesqp} Algebraic degrees and boundary degrees  for the QP problem~\eqref{eq:qp1}.}
  \end{table}

\begin{theorem} \label{thm:betaqp}
Let $m \leq n$ and suppose that $f_1,\ldots,f_m$ are generic polynomials in $\RR[x]_{\leq 2}$.
The algebraic boundary of the SDP-exact region $\mathcal{R}_{\bf f}$ is a hypersurface whose degree~equals
\begin{equation}
\label{eq:bdrdegQP}
\beta_{QP}(m,n) \,\,\, = \,\, \, 2^m\left( n  \binom{n}{m} - \binom{n}{m+1}\right).
\end{equation}
\end{theorem}

\Cref{tab:degreesqp} illustrates \eqref{eq:algdegQP0} and \Cref{thm:betaqp}.
It shows the algebraic degrees of quadratic programming and corresponding degrees of  rank-one boundaries.
Compare with \Cref{tab:degrees}. The diagonal entries $(m=n)$ in \Cref{tab:degreesqp}
are similar to those in the Max-Cut Problem (\Cref{ex:maxcut}),
but there is an index shift because the general objective function $g(x)$
is not homogeneous. We have $\beta_{QP}(n,n) = 2^n \cdot n$, since
the $n$ quadrics $\{f_i(x) = 0\}$ intersect in $2^n$ points,
and each of these contributes a spectrahedron of degree $n$ to the SDP-exact region.\\

For the proof we shall use polynomial ideals as in Section 2, but now
the ambient ring is $T = \CC[y_{00},y_{01},\ldots,y_{nn}, x_0,\ldots,x_n]$.
Using this variable ordering, we fix the lexicographic monomial order on $T$.
In particular, $y_{00}$ is the highest variable.
Let $I_{BV}=\operatorname{Min}_n\bigl(Y) + \langle Y x \rangle$
be the ideal generated by the $\binom{n+2}{2}$ minors
of $Y$ of size $n$ and the $n+1$ entries of vector $Y x$.

\begin{lemma}
The initial ideal ${\rm in}(I_{BV})$ is radical. It is minimally generated by
$\binom{n+2}{2} + \sum_{t=0}^{n-2} \binom{n+1}{t+1}$ squarefree monomials,
namely the leading terms of the $n \times n$ minors of $Y$,
and the monomials $\,x_t \cdot  y_{0 k_0} y_{1 k_1} \cdots y_{t k_t}\,$
where $\,t \in \{0,1,\ldots,n-2\}\,$ and $\, 0 \leq k_0 < k_1 < \cdots < k_t  \leq n$.
\end{lemma}

\begin{proof} It is well-known in commutative algebra that
the $n \times n$ minors of $Y$ form a reduced Gr\"obner basis. We augment these
to a reduced Gr\"obner basis for $I_{BV}$ by adding the entries of the row vector $x^T \tilde Y$ where
$\tilde Y$ is a certain matrix with $n+1$ rows and many more columns.
To construct this, we consider the $T$-module spanned by any subset of columns of $T$.
The {\em circuits} in such a submodule  of $T^{n+1}$ are the nonzero vectors with minimal support.
We consider all circuits whose support is a terminal segment $\{t,t{+}1,\ldots,n,n{+}1\}$.
The columns of $\tilde Y$ are all such circuits.
 These are formed by applying Cramer's rule to submatrices of $Y$
with row indices $0,\ldots,t-1$ and $t+1$ arbitrary columns.
The resulting entries of $x^T \tilde Y$ lie in $I_{BV}$.
They are linear in $x$, of degree $t+1$ in $Y$, and have the desired initial monomials.
One checks that their S-pairs reduce to zero, and that this Gr\"obner basis is reduced.
\end{proof}

\begin{corollary} \label{cor:happy}
The ideal $I_{BV}'$ obtained from $I_{BV}$ by eliminating the highest variable $y_{00}$ is generated
by those $n$ entries of $Y x$ and $n+1$ minors of $\,Y$ of size $n$ that do not use~$y_{00}$.
\end{corollary}
\begin{proof}
The elimination ideal $I'_{BV}$ is generated by elements of the lexicographic Gr\"obner basis
 that do not contain $y_{00}$. These are elements whose leading monomials
do not contain $y_{00}$. Each of these is a polynomial linear combination of
the above $2n+1$ generators of $I_{BV}$.
\end{proof}

\begin{proof}[Proof of \Cref{thm:betaqp}]
Let $N = \binom{n+2}{2}-1$.
As in the proof of \Cref{thm:beta}, we identify $CV$ with its
preimage in $\PP^{n} \times \PP^N$, that is,
 $ CV = \{ (x,Y) \;|\; Y x=0, \rank(Y) \leq n\}$.
   Its image $CV'$ under $\pi$ lives in $\PP^{n} \times \PP^{N-1}$.
The boundary $BV'$ is the projection of $BV$ into
 $\PP^{n} \times \PP^{N-1}$.

In \Cref{thm:beta}, the boundary was found by
intersecting $CV$ with the divisor given by the minor $c_{00}$ of $Y$,
and by removing the non-reduced excess component $\{x_0^2 = 0\}$.
In the present case, we still have that excess component, but it is reduced,
given by $x_0 = 0$.
The class $[\{x_0=0\}]$ is half of the pullback
of the hyperplane class $s$ of $\PP^{n}$.
Using~\eqref{eq:algdegQP0}, this~implies
\[ [BV'] \,=\,  \bigl( \,- \frac{1}{2}s \,+ \, n t \,\bigr) \, [CV']   \,\, = \,\,
  \sum_{m=1}^{n} \beta_{QP}(m,n) \cdot  s^{\binom{n}{2}-(m+1)} t^{m+1} . \]
The coefficients $\beta_{QP}(m,n)$ of this binary form are the combinatorial expressions in \eqref{eq:bdrdegQP}.

To see that the excess component is now  $\{ x_0 = 0\}$, we argue as follows.
Let $C'=(c_{0j})$ be the leftmost column of the adjugate matrix of $Y$.
Consider the ideals $I' := I_{CV}' + \langle c_{00} \rangle$ and
$J':= I_{CV}' + \langle C' \rangle \cdot \langle x_0 \rangle$.
We claim that $I'=\operatorname{sat}(J')$.
Observe that the $(n{+}1) \times 2$ matrix $ \bigl( \,x \,| \,C' \bigr)$ satisfies $\operatorname{Min}_2\bigl(\,x\,|\,C'\bigr) \subset I_{CV}'$.
This implies $c_{0j} x_0 \in J'$ for all $j \geq 1$.
Then $J' \subseteq I'$ and since $I'$ is saturated, $\operatorname{sat}(J') \subseteq I'$.  The reverse inclusion is implied by $c_{00} \in \operatorname{sat}(J')$, which follows from the fact that $c_{00}x_j \in \operatorname{Min}_2\bigl(\,x\,|\,C'\bigr) + \langle C' \rangle \cdot \langle x_0 \rangle$.
By \Cref{cor:happy}, the elimination ideal  is $I_{BV}' = I_{CV}' + \langle C' \rangle$.
So we may conclude that $CV' \cap \{c_{00}\} = BV' \cup (CV' \cap \{x_0=0\})$.
\end{proof}

\section{Bundles of Spectrahedral Shadows}\label{s:4}

We fix $\mathbf{f}=(f_1,\dots,f_m)$ as before.
For any $u\in \RR^n$ we consider the following two problems:
\begin{itemize}
  \item {\bf Linear Objective (Lin):} \ \ \  Minimize $u^T x$ subject to $x\in V_{\bf f}$. \vspace{-0.1in}
  \item {\bf Euclidean Distance (ED):} Minimize $\|x-u\|^2$ subject to $x\in V_{\bf f}$.
\end{itemize}
These problems are special instances of the quadratic program \eqref{eq:qp1}, with the cost matrices
\begin{equation} \label{eq:distanceC}
  \bigC^{\it lin}_u =
  \text{ \begin{footnotesize}\(
  \begin{pmatrix}
  0 & u_1 & u_2 &  \cdots & u_{n} \, \\
  u_1 & 0 & 0 & \cdots & 0 \\
  u_2 & 0 & 0 &  \cdots & 0 \\
  \vdots &  \vdots  & \vdots & \ddots & \vdots \\
  u_{n} & 0 & 0 & \cdots & 0 \smallskip \\
  \end{pmatrix}
  \)\end{footnotesize} }
  \qquad {\rm and} \qquad
  \bigC^{\it ed}_u =
  \text{ \begin{footnotesize}\(
  \begin{pmatrix}
  0 & -u_1 & -u_2 &  \cdots & -u_{n} \, \\
  -u_1 & 1 & 0 & \cdots & 0 \\
  -u_2 & 0 & 1 &  \cdots & 0 \\
  \vdots &  \vdots  & \vdots & \ddots & \vdots \\
  -u_{n} & 0 & 0 & \cdots & 1 \smallskip \\
  \end{pmatrix}.
  \)\end{footnotesize} }
\end{equation}
We write $\mathcal{R}^{\it lin}_{\bf f}$ and $ \mathcal{R}^{\it ed}_{\bf f} $ for the SDP-exact regions in $\RR^n$ of these two problems.
They are the intersections of $\mathcal{R}_{\bf f}$ with the affine subspaces
of $\mathcal{S}^{n+1}$ given in \eqref{eq:distanceC}.
The punchline of this section is that both regions are {\em normal bundles of spectrahedral shadows} over $V_{\bf f}$.
Namely, we shall write $\mathcal{R}^{\it lin}_{\bf f}$ and $\mathcal{R}^{\it ed}_{\bf f}$ as a union of
spectrahedral shadows, one for each point ${x} \in V_{\bf f}$.

The lower right block of $\bigC^{\it lin}$ and $\bigC^{\it ed}$ is independent of~$u$,
and thus the Hessian matrix ${\rm H}(\lambda)$ is independent of~$u$.
The spectrahedron defined by the constraint  $ {\rm H}(\lambda) \succ 0 $ is as follows:
\begin{align} \label{eq:masterspec}
  {\rm S}^{\it lin}_{\bf f} =
  \biggl\{ \, {\lambda}\in\RR^{m} \,:\, \sum_{i=1}^m \lambda_i A_i \,\prec \,0 \,\biggr\}
  \qquad {\rm and} \qquad
  {\rm S}^{\it ed}_{\bf f} =
  \biggl\{ \, {\lambda}\in\RR^{m} \,:\, \sum_{i=1}^m \lambda_i A_i \,\prec \,\id_n \,\biggr\}.
\end{align}
The sets in \eqref{eq:masterspec} are called \emph{master spectrahedra}.
Observe that ${\rm S}^{\it lin}_{\bf f}$ is a cone in~$\RR^{m}$.
Also note that ${\rm S}^{\it ed}_{\bf f}$ is full-dimensional because
$\lambda = (0,\ldots,0)$ is an interior point.
Let ${\rm Jac}_{\bf f}$ denote the Jacobian matrix of ${\bf f}$.
This matrix has format $n \times m$, and
its entry in row $ i$ and column $j$ is the linear polynomial $\partial f_j/\partial x_i$.
At any point $x \in V_{\bf f}$, the specialized Jacobian matrix ${\rm Jac}_{\bf f}(x)$ defines a
linear map $\RR^m \to \RR^n$, whose range is the \emph{normal space} of the
variety $V_{\bf f}$ at~$x$.
We consider all the images of the respective master spectrahedra under these linear maps.

\begin{theorem}\label{thm:unionspectrahedron}
  The SDP-exact regions for {\bf (Lin)} and {\bf (ED)} are comprised of
  the images of the corresponding master spectrahedra in the normal spaces of the variety $V_{\bf f}$.
  To be precise,
  \begin{equation*} %\label{eq:shadowf}
    \mathcal{R}_{\bf f}^{\it lin} =
    \bigcup_{{x} \in V_{\bf f}} \bigl( \tfrac{1}{2} \,{\rm Jac}_{\bf f}({x}) \cdot \mathrm{S}^{\it lin}_{\bf f} \, \bigr)
    \qquad {\rm and} \qquad
    \mathcal{R}_{\bf f}^{\it ed} =
    \bigcup_{{x} \in V_{\bf f}} \bigl(\, x - \tfrac{1}{2}\,{\rm Jac}_{\bf f}({x}) \cdot \mathrm{S}^{\it ed}_{\bf f} \, \bigr).
  \end{equation*}
  Moreover, the above unions are disjoint because our spectrahedra are
  relatively open.
\end{theorem}
\begin{proof}
  The result follows by substituting \eqref{eq:distanceC} into
  \Cref{def:rank1sol2}.  Disjointness holds because any $u$ in one of the
  parenthesized sets has the associated $x$ as its  unique optimal solution.
  \end{proof}

One consequence of \Cref{thm:unionspectrahedron} is that the SDP-exact region for an ED~problem is always full-dimensional.
This fact was observed in~\cite{CAPT}, where it was shown to have interesting applications
in computer vision, tensor approximation and rotation synchronization.

\begin{corollary}
  If $x $ is a regular point of $V_{\bf f}$,
  then $\mathcal{R}^{\it ed}_{\bf f}$ contains an open neighborhood of $x$.
\end{corollary}

\begin{proof}
  The regularity hypothesis means that
  $\rank({\rm Jac}_{\bf f}(x))= \codim_x(V_{\bf f})$.
  This ensures that  ${\rm Jac}_{\bf f}({z})\cdot {\rm S}^{\it ed}_{\bf f}$ is full-dimensional in the normal space of $V_{\bf f}$
  at any point $z$  close to~$x$.
\end{proof}

For finite complete intersections, the SDP-exact regions are finite unions of spectrahedra:

\begin{corollary}\label{thm:finitevariety0}
  Let ${\bf f}=(f_1,\dots,f_n)$ be a complete intersection with $k\leq 2^n$ real points.
  Then
  \begin{itemize}
    \item[(a)] $\mathcal{R}^{\it lin}_{\bf f}$ consists of $k$ spectrahedral cones, each of them
    isomorphic to the master~${\rm S}^{\it lin}_{\bf f}$. \vspace{-0.05in}
        \item[(b)] $\mathcal{R}^{\it ed}_{\bf f}$ consists of $k$ full-dimensional spectrahedra,
    each  isomorphic to the master~${\rm S}^{\it ed}_{\bf f}$.
  \end{itemize}
\end{corollary}

\begin{proof}
The linear map ${\rm Jac}_{\bf f}({x})$ is injective and hence invertible on its image.
Therefore, the spectrahedral shadow ${\rm Jac}_{\bf f}({x})\cdot {\rm S}_{\bf f}$ is  actually a spectrahedron,
 linearly isomorphic to ${\rm S}_{\bf f}$.
\end{proof}

\begin{figure}[htb]
 \centering
 \includegraphics[height=130pt]{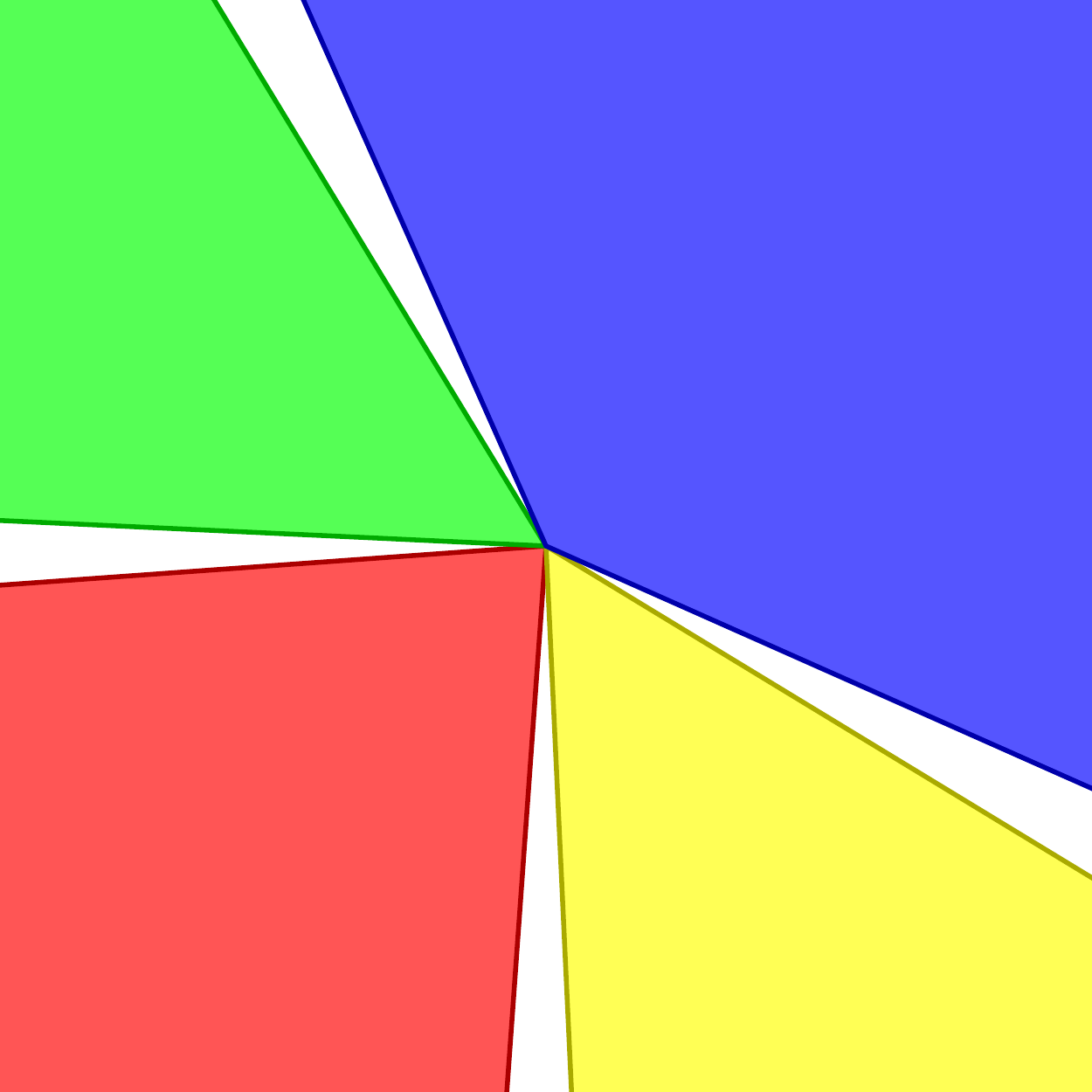}
 \hspace{60pt} \quad
 \includegraphics[height=130pt]{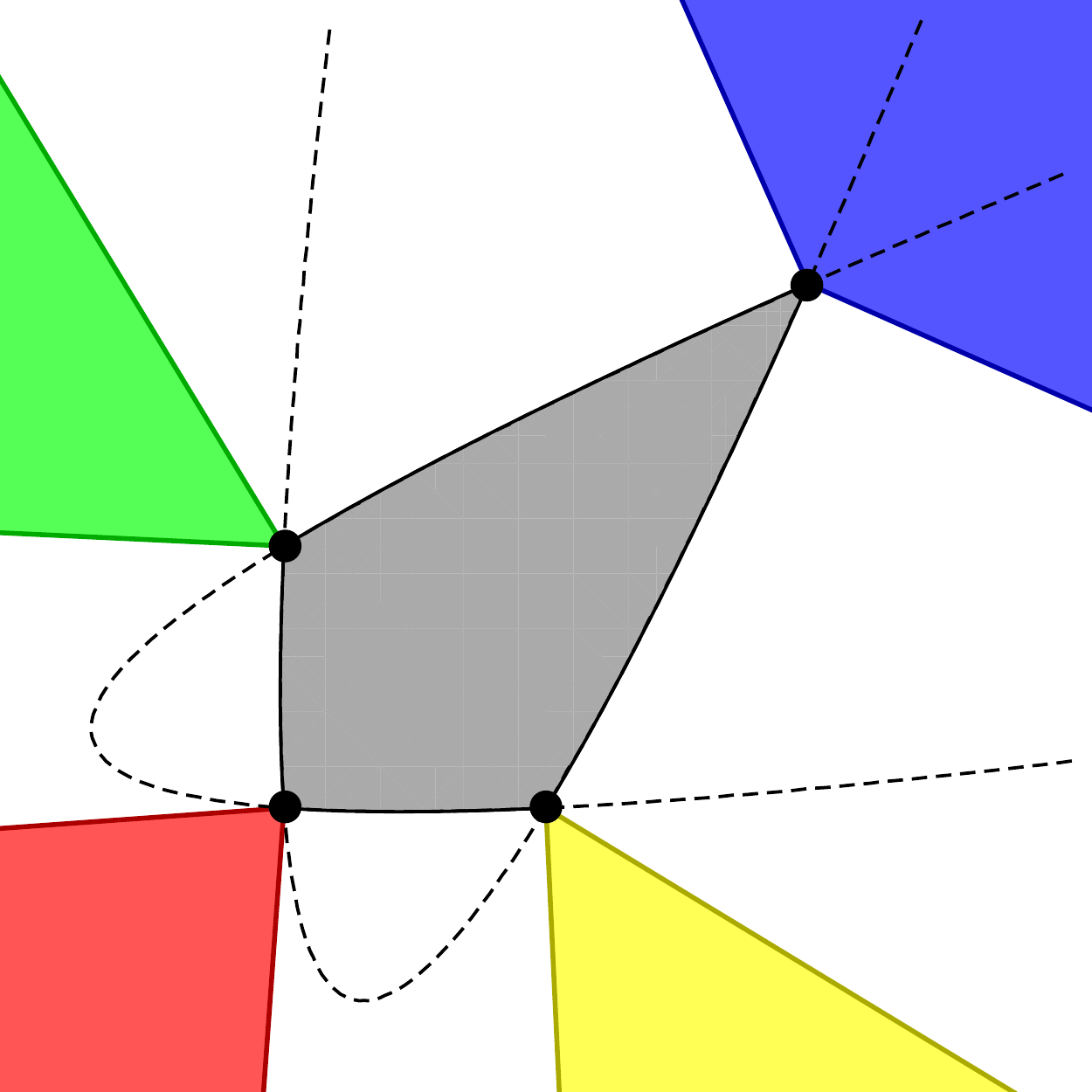}
 \caption{   We consider two quadrics in $\RR^2$ that meet in four points.
   The SDP-exact region for minimizing linear functions
   over this intersection consists of four cones, shown in the left.
   These are the normal cones at the first theta body ${\rm TH}_1({\bf f})$, as illustrated in the right.
   \label{fig:thetabody}}
\end{figure}

The spectrahedral cones in $\mathcal{R}^{\it lin}_{\bf f}$ are tightly connected to the first {\em theta body} of $\langle {\bf f}\rangle$, denoted ${\rm TH}_1({\bf f})$, introduced by Gouveia et al.~in \cite{GPT}.
The theta bodies of ${\bf f}$ are tractable approximations to the convex hull of~$V_{\bf f}$, whose construction relies on the Lasserre hierarchy~\cite{BPT,Lasserre2010,Lasserre2001}.
Later in this section we will show that $\mathcal{R}^{\it lin}_{\bf f}$ consists of the normal cones of~${\rm TH}_1({\bf f})$.

\begin{example}[$m=n=2$]\label{ex:fourpoints} \rm
Consider two quadrics in two variables such that $V_{\bf f}$
consists of four points in convex position in $\RR^2$. The region
 $\mathcal{R}^{\it ed}_{\bf f}$ was illustrated in  \Cref{fig:voronoi2d}.
  The region $\mathcal{R}^{\it lin}_{\bf f}$ consists of four cones
  that sit inside the normal cones at the quadrilateral ${\rm conv}(V_{\bf f})$.
We explain this for the specific instance examined in \cite[Example~5.6]{GPT}:
  \begin{align*}
   {\bf f}  =
   (x_1 x_2-2 x_2^2+2 x_2,\, x_1^2-x_2^2-x_1+x_2),\qquad
  V_{\bf f} =
    \{\, (0,0)\,,\,(0,1)\,,\,(1,0)\,,\,(2,2) \,\} .
  \end{align*}
  The first  theta body ${\rm TH}_1({\bf f})$ is seen in \cite[Figure~3]{GPT}.
  Our rendition in  \Cref{fig:thetabody} show also the SDP-exact region $\mathcal{R}^{\it lin}_{\bf f}$.
  It consists of the normal cones of ${\rm TH}_1({\bf f})$
  at the four points in $V_{\bf f}$.
  For more details see \Cref{thm:thetabody}.
\end{example}

It is interesting to examine \Cref{thm:finitevariety0} (b) when
$m=n$ and $V_{\bf f}$ consists of $2^n$ real points.
We know that $\mathcal{R}^{\it ed}_{\bf f}$ consists of $2^n$ full-dimensional spectrahedra
of degree $n$. We show that these hypersurfaces are pairwise tangent, and also tangent
to the walls of the Voronoi diagram. The case $n=2$ was seen in  \Cref{fig:voronoi2d},
whereas the case $n=3$ is shown in \Cref{fig:somosas}.

\begin{figure}[htb]
 \centering
 \includegraphics[width=250pt]{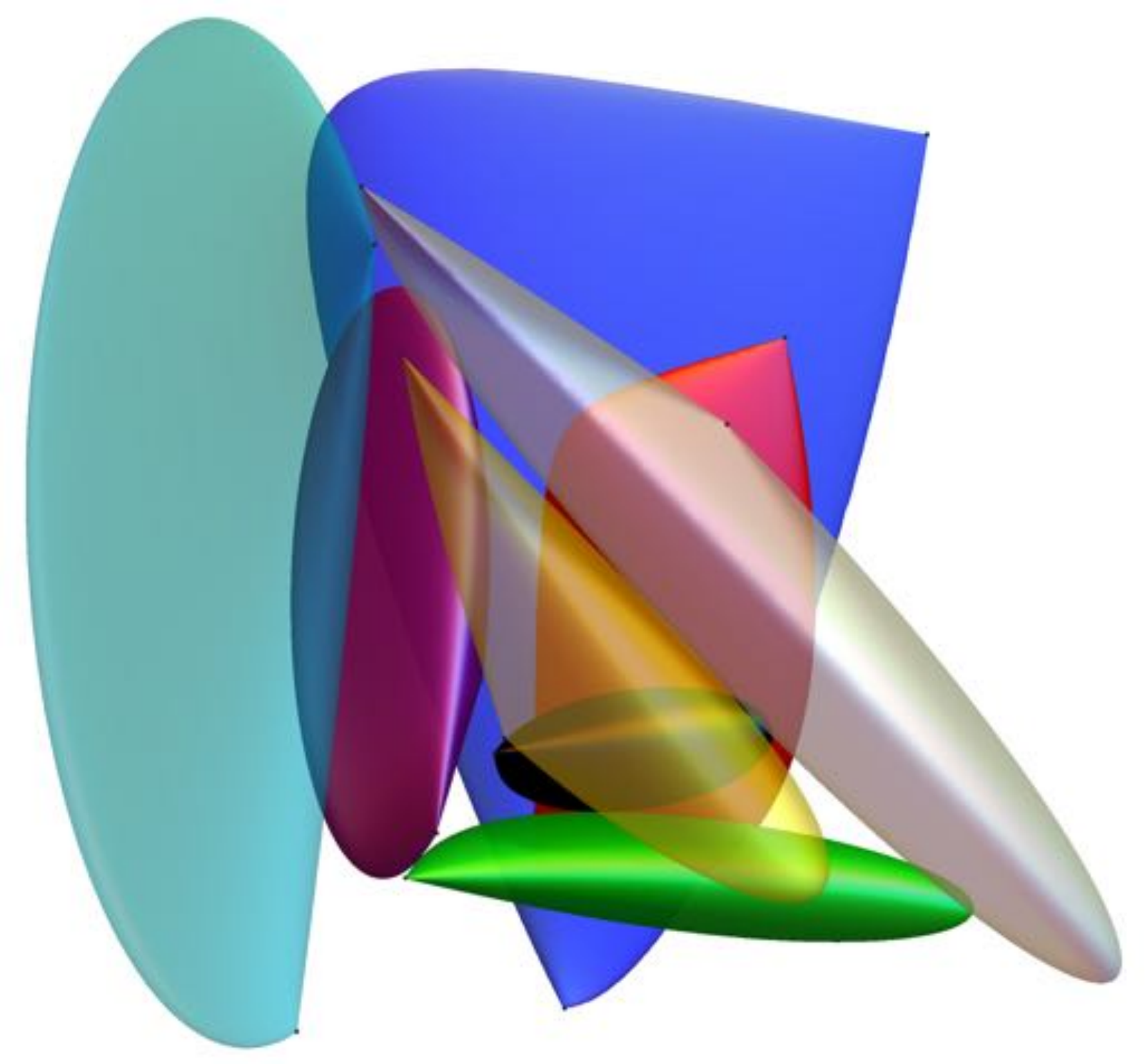}
 \caption{
   We consider
   three quadrics in $\RR^3$ that meet in eight points.
   The SDP-exact region for the ED~problem
   on this variety consists of eight spectrahedra, each around one of these points.
   The algebraic boundaries of the spectahedra are pairwise tangent.
   }
   \label{fig:somosas}
\end{figure}

For $ x \in V_{\bf f}$, we set $\,S_x =  x - \tfrac{1}{2}\,{\rm Jac}_{\bf f}({x}) \cdot \mathrm{S}^{\it ed}_{\bf f} \,$
and we write $\,\partial_{\rm alg} S_x$ for its algebraic boundary.

\begin{theorem} \label{thm:finitevariety} Let $m=n$ and ${\bf f}$ generic, so $V_{\bf f}$ is finite.
Let  $x,x' \in V_{\bf f}$, and $S_x, S_x'$ be the corresponding spectrahedra, and let $\mathrm{bsc}\subset \RR^{n}$ be the  bisector hyperplane of $x$ and~$x'$.     There is a point $u\in\RR^{n}$ at which the three hypersurfaces
     $\mathrm{bsc}$, $\partial_{\rm alg} S_x$ and $\partial_{\rm alg} S_x'$ meet tangentially.
\end{theorem}

\begin{proof}
  Let $p(\lambda):= \det(\id_{n}-\sum_i \lambda_i {A}_i)$ be the defining polynomial of $\partial_{\rm alg} {\rm S}^{\it ed}_{\bf f}$.
  Then    $\,p_x(u):= p(2\,{\rm Jac}_{\bf f}({x})^{-1} u - x)\,$ is the defining polynomial of~$\partial_{\rm alg} S_x$.
  We shall construct  a point $u_x$ in the hypersurface $ \partial_{\rm alg} S_x$ whose
     normal vector    $\nabla_u p_x (u_x) $ is parallel to $x-x'$. Notice that
  \begin{align*}
    \nabla_{u} p_x
  \,  =\, 2 (\nabla_\lambda p) {\rm Jac}_{\bf f}({x})^{-1}
  \,  = \,- 2 \left( {A}_1 \bullet M, \dots, {A}_m \bullet M \right) \cdot {\rm Jac}_{\bf f}({x})^{-1},
  \end{align*}
  where    $M $ denotes the adjugate of $\id_{n} - \textstyle\sum_i {\lambda}_i {A}_i$.
  Since this matrix is supposed to be singular,
    \begin{align}\label{eq:adjugate}
    (\id_n-\textstyle\sum_i \lambda_i {A}_i) M  = 0, \quad
    \left( {A}_1 \bullet M, \dots, {A}_m \bullet M \right) \propto
    \tfrac{1}{2} (x' - x)^T {\rm Jac}_{\bf f}({x}), \quad
    \rank (M) = 1.
  \end{align}
  We claim that $M = (x'-x)(x'-x)^T$ satisfies the constraint in the middle.
  This is seen by showing that the $i$-th coordinate of the vector $\tfrac{1}{2} (x' - x)^T {\rm Jac}_{\bf f}({x})$ equals
  \begin{equation}
  \label{eq:Jrankone}
  \begin{aligned}
    (x' - x)^T &(a_i + {A}_i x)
    =  x'^T {A}_i x + a_i^T(x' - x) - x^T {A}_i x \\
    &= x'^T {A}_i x -\tfrac{1}{2}(x'^T {A}_i x' - x^T {A}_i x) - x^T {A}_i x
    = {A}_i \bullet (-\tfrac{1}{2}) (x'-x) (x'-x)^T.
  \end{aligned}
  \end{equation}
  The desired vector $\lambda$ is then determined by the equation $(\id_n-\sum_i\lambda_i {A}_i)(x'-x)=0$.
Now,  \eqref{eq:adjugate} holds, and the point    $u_x = x - \tfrac{1}{2}{\rm Jac}_{\bf f}({x})(\lambda)$
has its normal at $ \partial_{\rm alg} S_x$ parallel to $x' - x$.

  We similarly construct $u_x' \in \partial_{\rm alg} S_x'$.
    By~\eqref{eq:Jrankone}, we have  $(x - x')^T {\rm Jac}_{\bf f}({x}') = (x' - x)^T {\rm Jac}_{\bf f}({x})$.
   Hence the value of $M$ that satisfies~\eqref{eq:adjugate} is the same for both $x$ and~$x'$, and thus $u_x=u_x'$.

  Finally, let us show that $u_x$ lies on $\mathrm{bsc}$.
Since  $(\id_{n} - \textstyle\sum_i \lambda_i {A}_i) (x'-x) = 0$, we have
  \begin{align*}
    u_x^T (x'-x)
    &\,\,=\,\, (x - \textstyle\sum_i \lambda_i(a_i + {A}_i x))^T (x'-x) \\
    &\,\,=\,\, -(\textstyle\sum_i \lambda_i  a_i^T) (x'-x)
    + x^T (\id_{n} - \textstyle\sum_i \lambda_i {A}_i) (x'-x)
   \,\, =\,\, -\textstyle\sum_i \lambda_i  a_i^T (x'-x).
  \end{align*}
The difference $\| u_x - x' \|^2 - \| u_x - x \|^2 $ equals
  \begin{align*}
    \|x'\|^2 &- \|x\|^2 - 2 u_x^T (x'-x)
    = x'^T x' - x^T x + 2\textstyle\sum_i {\lambda}_i a_i^T(x'-x) \qquad \quad \\
    &= x'^T x' - x^T x - \textstyle\sum_i {\lambda}_i (x'^T {A}_i x' - x^T {A}_i x)
    = (x' + x)^T (\id_{n} - \textstyle\sum_i {\lambda}_i {A}_i ) (x' - x)
    =\,\,\, 0.
  \end{align*}
  We see that $u_x$ is equidistant from $x$ and~$x'$, i.e.,~$u_x$ belongs to the
   hyperplane $\mathrm{bsc}$.
   We have shown that our three hypersurfaces all pass through $u_x$ and have the
   same normal vector.
\end{proof}

We next illustrate how the normal bundle from \Cref{thm:unionspectrahedron} looks for a curve.

\begin{figure}[htb]
 \centering
 \includegraphics[height=180pt]{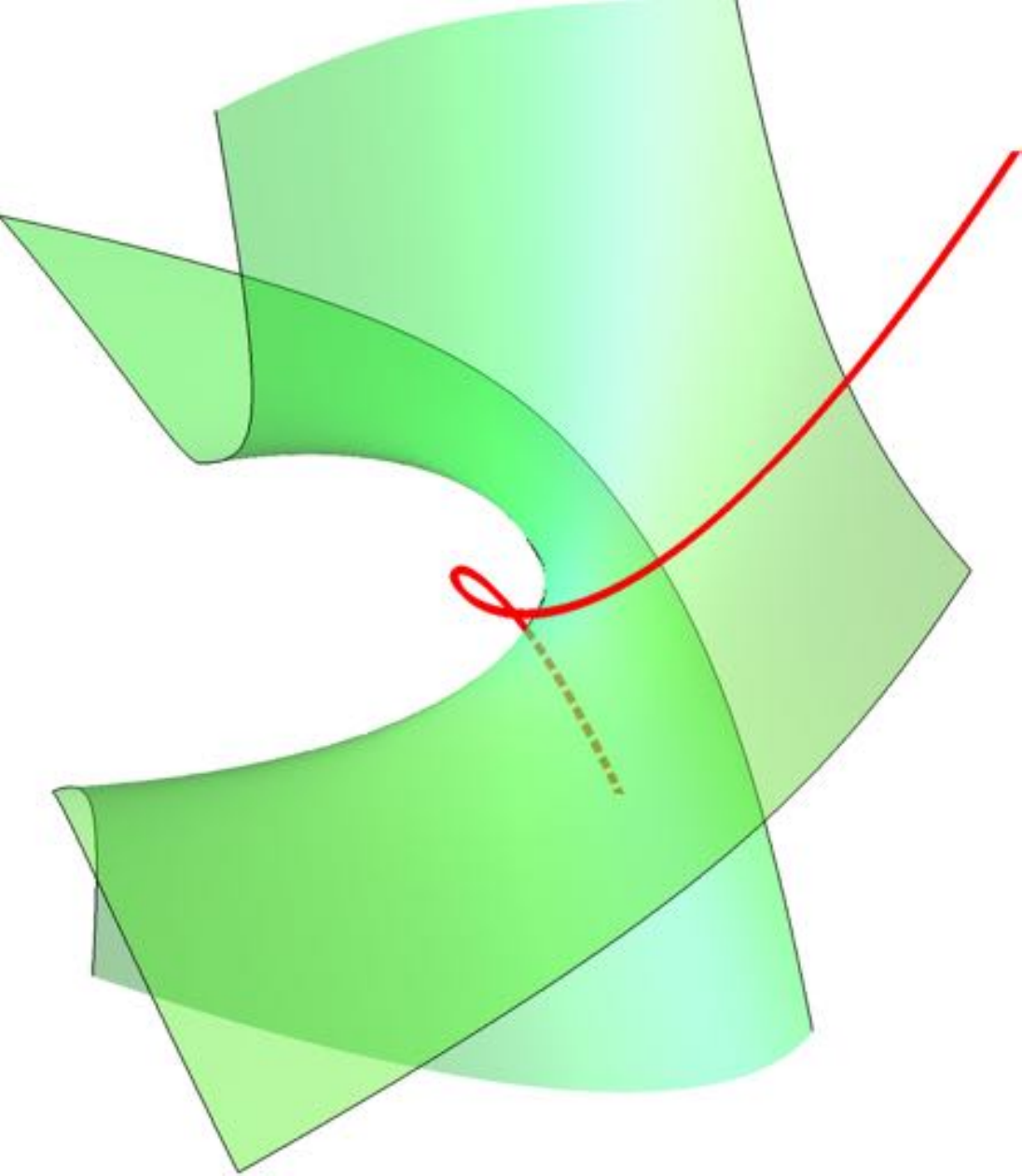}\hspace{50pt}%
 \includegraphics[height=180pt]{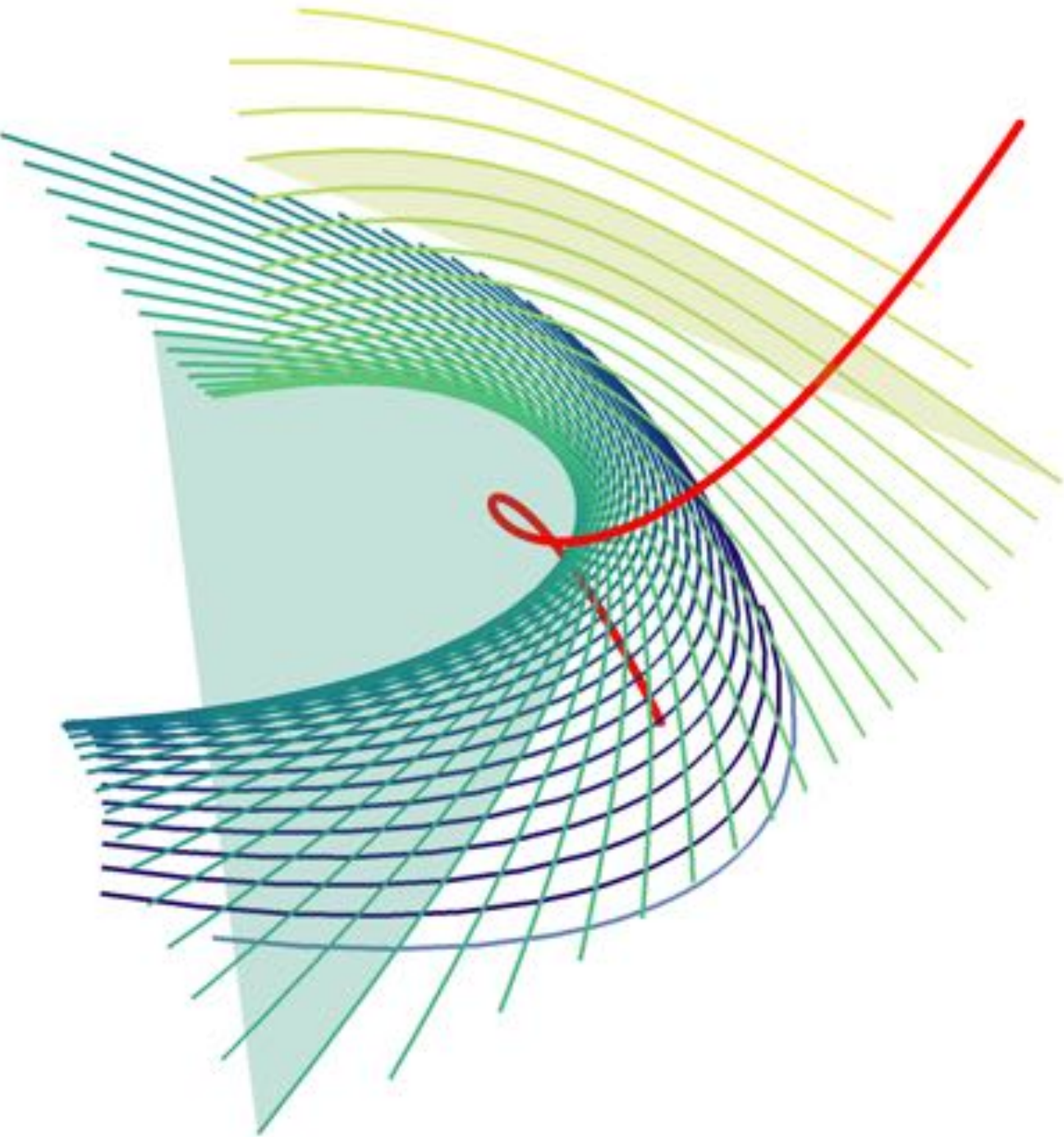}
 \caption{The boundary of the SDP-exact region for the ED~problem on the twisted cubic curve
 is ruled by parabolas. This surface has degree eight.
 It is computed in \Cref{ex:twistedcubic1}.
 \label{fig:twistedcubic}
 }
\end{figure}

\begin{example}\label{ex:twistedcubic0} \rm
  Let ${\bf f} = (x_2 - x_1^2,\, x_3 - x_1 x_2)$, so $V_{\bf f}$ is the {\em twisted cubic curve} in $\RR^3$.
  This specific instance was examined in \cite[Example~1.1]{CAPT}.
  The spectrahedron
  $\mathrm{S}^{\it ed}_{\bf f}$ is the interior of a parabola, namely $\{ \lambda_2^2 < 2 \lambda_1 + 1 \}$.
  The image ${x} - \tfrac{1}{2}{\rm Jac}_{\bf f}({x}) \cdot {\rm S}^{\it ed}_{\bf f}$ is a parabola in the normal plane at~$x$.
  The boundary $\partial \mathcal{R}^{\it ed}_{\bf f}$ is the union of all these parabolas, as shown in \Cref{fig:twistedcubic}.
\end{example}

We will elaborate more on the ED~problem in \Cref{s:5}.
To conclude this section, we briefly develop the connection between our SDP-exact region $\mathcal{R}^{\it lin}_{\bf f}$
 and the theory of theta bodies~\cite{GPT}.
 By \cite[Lemma~5.2]{GPT}, the {first theta body} of our instance ${\bf f}$ is
\begin{align*}
  \mathrm{TH}_1(\mathbf{f}) \,\quad =
  \bigcap_{\substack{F \in \langle \mathbf{f}\rangle\\ F \text{ convex quadric}}}
\!\!\!\!\!  \bigl\{\,x\in \RR^n \, : \, F(x)\leq 0 \,\bigr\}.
\end{align*}
By \cite[\S 2]{GPT}, the set $\mathrm{TH}_1(\mathbf{f})$ is a
 spectrahedral shadow that contains the convex hull of~$V_{\bf f}$.

\begin{proposition} \label{thm:thetabody}
  Let $B = \mathrm{TH}_1(\mathbf{f})$ be the first theta body for the problem {\bf (Lin)}.
  Then the SDP-exact region $ \mathcal{R}^{\it lin}_{\bf f}  $ is the union of the normal cones to
  $B$ at all points in $V_{\bf f}$. In symbols,
  \begin{align*}
    \mathcal{R}^{\it lin}_{\bf f} \,\, = \bigcup_{x \in V_{\bf f}} \! N_B(x).
  \end{align*}
\end{proposition}
\begin{proof}
  Note that $u\in N_B(x)$ if and only if $x=\argmax_{y\in B} u^T y$.
  On the other hand, the problem $\max_{y\in B} u^T y$ is equivalent to the SDP relaxation of our QP~\eqref{eq:qp1}.
  Then,
  \begin{align*}
    \bigcup_{x \in V_{\bf f}} N_B(x)
    &\,\, =\,\,
    \{ u \in \RR^n : (\argmax_{y\in B} u^Ty) \in V_{\bf f} \}\\
    &\,\,=\,\,
    \{ u \in \RR^n : \text{ the solution of the SDP relaxation lies in } V_{\bf f} \}.
  \end{align*}
  By definition, this set   is the SDP-exact region for {\bf (Lin)}.
  For an illustration see \Cref{fig:thetabody}.
\end{proof}

\section{Boundary Hypersurfaces in \texorpdfstring{$\RR^n$}{Rn}}\label{s:5}

We now examine our degrees of the ED~problem.
Following \cite{DHOST}, the Euclidean distance degree of $V_{\bf f}$, denoted
$\operatorname{EDdegree}(V_{\bf f})$,
counts the number of complex critical points for the squared distance function
$g_u(x) = \| x-u \|^2$ on the variety $V_{\bf f}$, where $u \in \RR^n$ is a generic point.

\begin{proposition} \label{prop:EDdegree}
The algebraic degree of the quadratic program \eqref{eq:qp1}
that solves the ED~problem for $V_{\bf f}$ is $\operatorname{EDdegree}(V_{\bf f})$.
This is bounded above by  $ 2^{m} \binom{n}{m} $.
Equality holds for generic~${\bf f}$.
\end{proposition}

\begin{proof}
The first statement is immediate from the definition of the ED degree.
The last two statements follow from \cite[Proposition~2.6]{DHOST}.
\end{proof}

We next assume that ${\bf f}$ is generic.
Hence $V_{\bf f}$ is a generic complete intersection.
We are interested in the degree  $\beta_{ED}(m,n)$ of the hypersurface
$\,\partial_{\rm alg} \mathcal{R}^{\it ed}_{\bf f} \subset \RR^n\,$ that bounds
the SDP-exact region  for the ED~problem. \Cref{tab:degreesed} shows $\beta_{ED}(m,n)$ for
some small cases.

\begin{figure}[htb]
 \centering
 \begin{minipage}{0.59\textwidth}
   \centering
 \includegraphics[width=220pt]{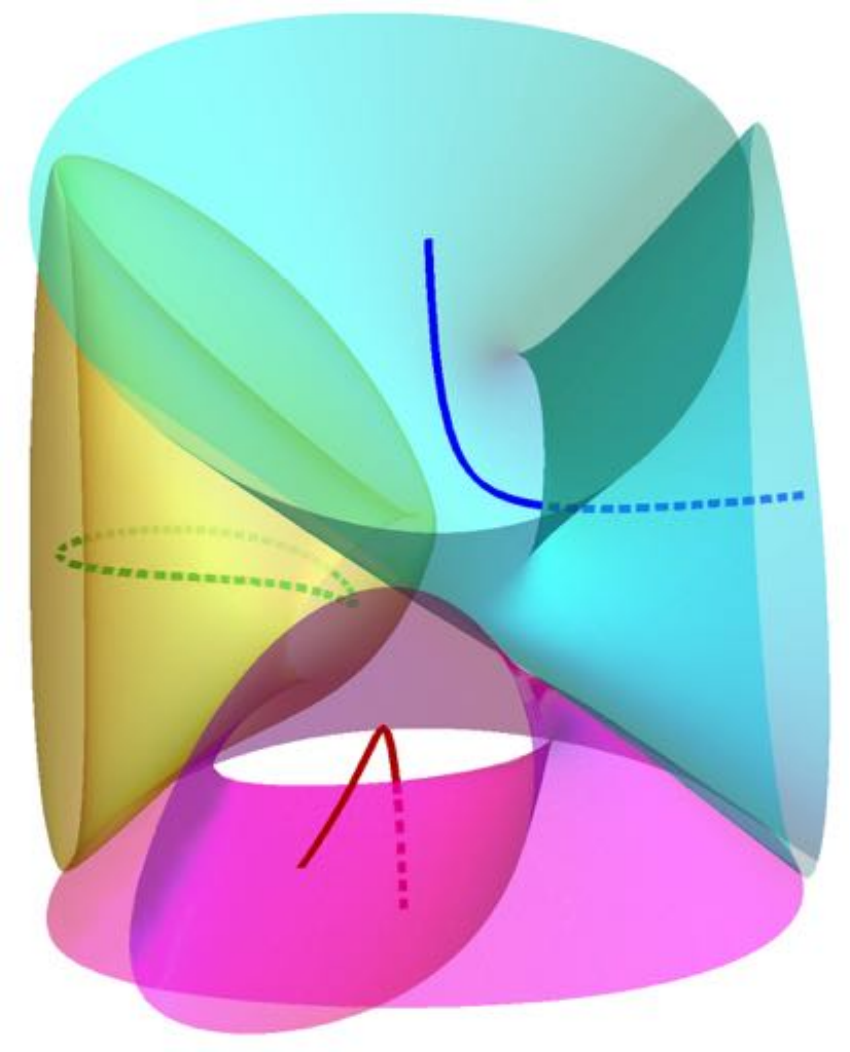}
 \end{minipage}%
 \begin{minipage}{0.4\textwidth}
   \centering
   \includegraphics[width=160pt]{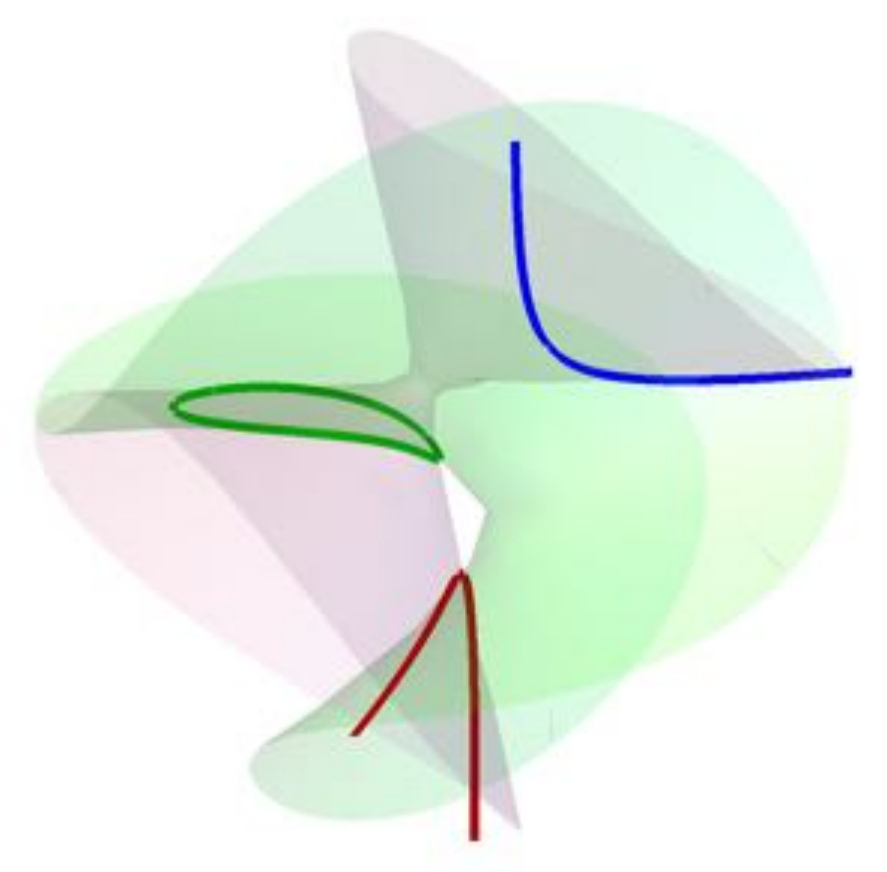} \\
 \includegraphics[width=90pt]{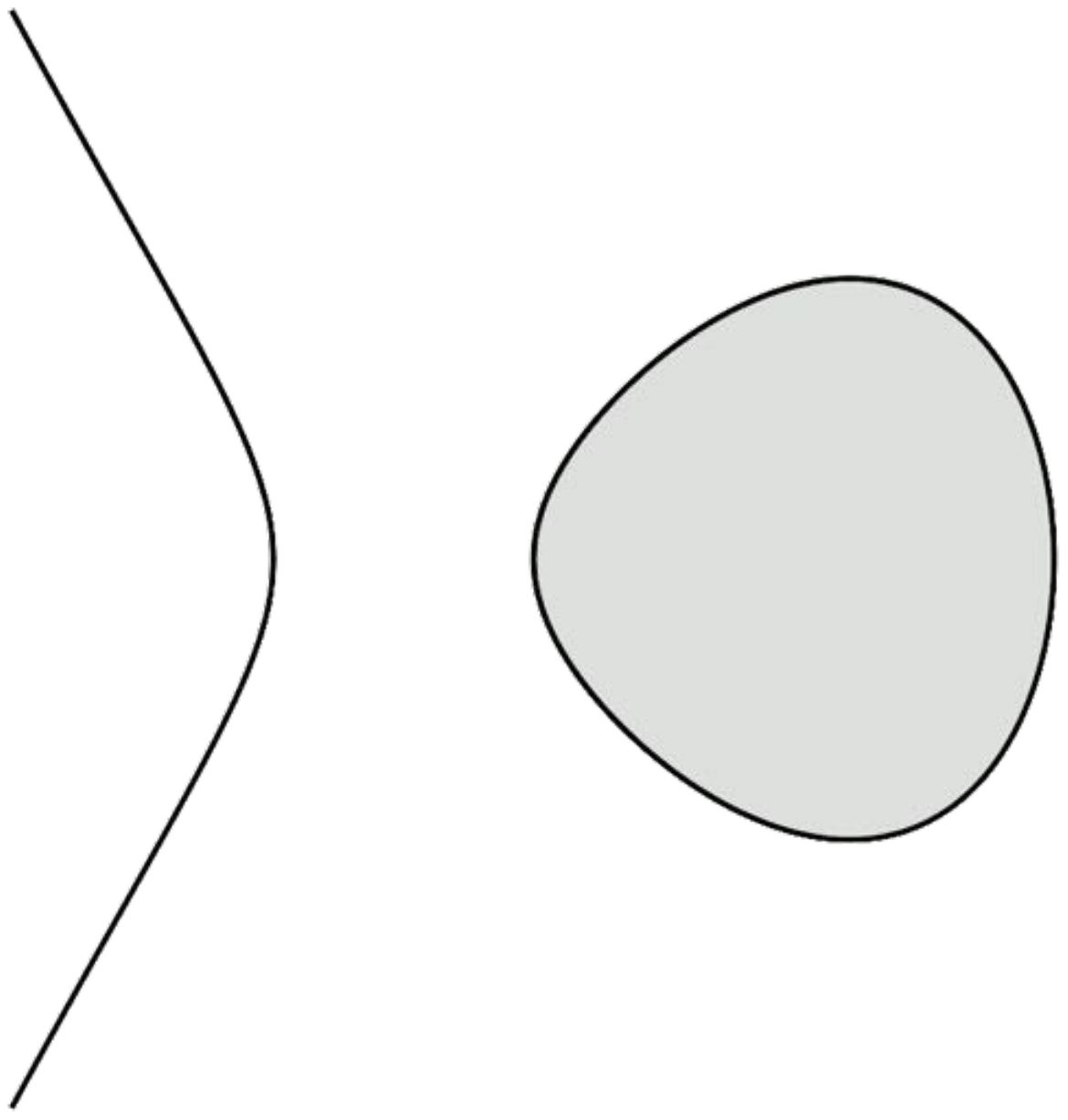}
 \end{minipage}%
 \caption{
   Upper right: Space curve cut out by two quadrics.
   Left: The SDP-exact region for its ED~problem.
      Lower right: The elliptic curve that defines the master spectrahedron.
 }
 \label{fig:hyperboloids}
\end{figure}

\begin{example}[$m=2,n=3$] \rm \label{ex:twoellipticcurves}
  \Cref{fig:hyperboloids} shows the SDP-exact region for a generic instance.
  Its boundary is an irreducible surface of degree $24$.
 The master spectrahedron is the convex region of a planar cubic
    (lower right in \Cref{fig:hyperboloids}).
  The variety $V_{\bf f}$ is a space curve of degree~$4$, obtained by intersecting  two hyperboloids
  (upper right in \Cref{fig:hyperboloids}).
  We regard both curves as elliptic curves, the first in $\PP^2$ and the second in $\PP^3$.
The product of these two elliptic curves is an {\em abelian surface}, which has
degree $24$ under its Segre embedding into $\PP^2 \times \PP^3 \subset \PP^{11}$.
Our boundary surface $\partial_{\rm alg} \mathcal{R}^{\it ed}_{\bf f}$ is a projection
of this surface into $\PP^3$. This explains $\beta_{ED}(2,3) = 24$.
The picture on the left in \Cref{fig:hyperboloids} shows
$\partial \mathcal{R}^{\it ed}_{\bf f}$ in real affine space $\RR^3$. Each of the three connected
components of the curve $V_{\bf f}$ is surrounded by
one color-coded component of that surface.  These three pieces of
$\partial \mathcal{R}^{\it ed}_{\bf f}$ are pairwise tangent along curves.
 \end{example}

\begin{table}[h]
 \begin{center} \qquad
\begin{tabular}{c|*{5}{c}}
\toprule
  \multicolumn{6}{c}{ED~degrees for $V_{\bf f}$}\\
\midrule
  $m\backslash n$ & 2& 3 & 4 & 5 & 6 \\
\midrule
2 & 4 & 12 & 24 & 40 & 60 \\
3 &   & 8 & 32 & 80 & 160 \\
4 &   &   & 16 & 80 & 240 \\
5 &   &   &   & 32 & 192 \\
6 &   &   &   &   & 64 \\
 \end{tabular}\qquad \qquad
\vspace{-0.11in}
\begin{tabular}{c|*{5}{c}}
\toprule
  \multicolumn{6}{c}{Boundary degrees $\beta_{ED}(m,n)$}\\
\midrule
  $m\backslash n$ & 2& 3 & 4 & 5 & 6 \\
\midrule
 2 & 8 & 24 & 48 & 80 & 120 \\
 3 &   & 24 & 96 & 240 & 480 \\
 4 &   &   & 64 & 320 & 960 \\
 5 &   &   &   & 160 & 960 \\
 6 &   &   &   &   & 384 \\

 \end{tabular}
\end{center}
\caption{\label{tab:degreesed} Algebraic degrees and boundary degrees for the ED~problem.}
  \end{table}

For the subsequent degree computations we record the following standard fact from algebraic geometry.
\Cref{ex:twoellipticcurves} used this formula for deriving the number $3 \cdot 4 \cdot \binom{1+1}{1} = 24$.

\begin{lemma}  \label{lem:degXxY} Fix two projective varieties $V \subset \PP^n$ and $W \subset \PP^m$.
The projective variety
 $V \times W$  has degree $\,\deg(V)\deg(W) \binom{\dim V + \dim W}{\dim W}\,$
in the Segre embedding of $ \,\PP^n \times \PP^m \,$ in $\,\PP^{(n+1)(m+1)-1}$.
\end{lemma}

We consider the product of our feasible set $V_{\bf f}$ with the algebraic boundary
of its master spectrahedron $ {\rm S}^{\it ed}_{\bf f}$.
This is the real algebraic variety
$\,V_{\bf f} \times \partial_{\rm alg} {\rm S}^{\it ed}_{\bf f} \,$ in $\RR^n \times \RR^m$.
We identify this variety with its Zariski closure in the product
of complex projective spaces $\PP^n \times \PP^m$.
Under the Segre map, we embed
$\,V_{\bf f} \times \partial_{\rm alg} {\rm S}^{\it ed}_{\bf f} \,$
as a projective variety in $\PP^{(m+1)(n+1)-1}$.

\begin{corollary} \label{cor:hasdimension}
  The variety $\,V_{\bf f} \times \partial_{\rm alg} {\rm S}^{\it ed}_{\bf f} \,$
has dimension $n -1$ and degree $m\, 2^m \binom{n}{m}$.
\end{corollary}

\begin{proof}
The variety $V_{\bf f}$ has dimension $n-m$ and degree $2^m$.
The variety  $\partial_{\rm alg} {\rm S}^{\it ed}_{\bf f} $ has dimension $m-1$ and degree $n$.
By \Cref{lem:degXxY}, their product has degree  $\,2^m \cdot n \cdot \binom{n-1}{m-1}
= m \cdot 2^m \cdot \binom{n}{m}$.
\end{proof}

By \Cref{thm:unionspectrahedron}, the boundary of the SDP-exact region is the image of
$\,V_{\bf f} \times \partial_{\rm alg} {\rm S}^{\it ed}_{\bf f} \,$ under
\begin{equation}
\label{eq:affpsi}
  \psi\,:\, \RR^n \times \RR^m \,\to\, \RR^n, \qquad
  (x,\lambda) \,\,\mapsto \,\,
  x - \tfrac{1}{2}{\rm Jac}_{\bf f}({x}) \lambda
  \,=\, x - \textstyle\sum_{i=1}^m \lambda_i (a_i + A_i x).
\end{equation}
The map $\psi$ is bilinear. We consider its homogenization
\begin{align} \Psi \,:\,
 \PP^n\times \PP^m\, \dashrightarrow \, \PP^n, \,\,\,
 \bigl(\,(x_0:x)\, ,\,  (\lambda_0:\lambda)\, \bigr) \,\mapsto \,   \bigl( \,\lambda_0 x_0\,: \,
  \lambda_0 x - \textstyle\sum_{i=1}^m     \lambda_i (x_0 a_i+  A_i x)\,   \bigr). \label{eqn:projPsi}
\end{align}
This map factors as the Segre embedding $\sigma$ followed by a linear projection $\pi$:
\begin{equation}
\label{eq:segrefollowed}
 \PP^n \times \PP^m \,\overset{\sigma}{\longrightarrow}\,\, \PP^{(n+1)(m+1)-1}
\,\overset{\pi}{\dashrightarrow}\,\, \PP^n .
\end{equation}

\begin{lemma} \label{lem:basepointfree}
  The restriction of $\,\pi$ to  (the image under $\sigma$ of) $\,V_{\bf f} \times \partial_{\rm alg} {\rm S}^{\it ed}_{\bf f} \,$
is  base-point free. \end{lemma}

\begin{proof}
  We show that $L \cap \sigma(V_{\bf f} \times \partial_{\rm alg} {\rm S}^{\it ed}_{\bf f} ) = \emptyset$, where
  $L \subset \PP^{(n+1)(m+1)-1}$ is the base locus of~$\pi$.
  By  \eqref{eqn:projPsi}, we know that $L$ is contained in $\{\lambda_0 x_0 = 0\}$.
  First, assume $\lambda_0=0$ and $x_0=1$.
  The equations from \eqref{eqn:projPsi} simplify to $\sum_{i=1}^{m}
  \lambda_i(a_i +  A_i x) = 0$, which means ${\rm Jac}_{\bf f}(x) \lambda = 0$.
  But this is impossible because ${\rm Jac}_{\bf f}(x)$ has full rank, by
  genericity of ${\bf f}$.
  Consider now the case $x_0=0$.
  We may assume that $m<n$, as otherwise $V_{\bf f}$ does not intersect $\{x_0=0\}$.
  Setting the image in \eqref{eqn:projPsi} to zero, we get $\lambda_0 x - \sum_{i=1}^{m} \lambda_i ( A_i x) = 0$.
  Viewed as a system of linear equations in $\lambda_0,\lambda_1,\ldots, \lambda_m $,
   this is overconstrained, so by genericity it has no nonzero solution.
   \end{proof}

   We now write $\pi$ for the restriction to $V_{\bf f} \times \partial_{\rm alg} {\rm S}^{\it ed}_{\bf f}$.
 \Cref{lem:basepointfree} and the dimension part in \Cref{cor:hasdimension}
 show that $\pi$ is a dimension-preserving morphism onto $\partial_{\rm alg} \mathcal{R}^{\it ed}_{\bf f}$.
 The degree of this morphism, denoted  $\deg(\pi)$, is the cardinality of the fiber of $\pi$
 over a generic point in the image.
  By \cite[Proposition~5.5]{Mum76}, the degree of the source
  equals the degree of the image times the degree of the map.
Hence, \Cref{lem:degXxY} implies the following result:

\begin{theorem}  \label{thm:betaED}
The degree of the algebraic boundary $\partial_{\rm alg} \mathcal{R}^{\it ed}_{\bf f}$ of the
SPD-exact region is
\begin{equation*} %\label{eq:degpi}
  \beta_{ED}(m,n) \quad =\quad \frac{1}{\deg(\pi)} \cdot m\, 2^{m} \binom{n}{m}.
\end{equation*}
 \end{theorem}

We conjecture that $\deg(\pi)=1$ whenever our variety $V_{\bf f}$ is not a hypersurface, i.e.,~whenever
 $m \geq 2$. This was verified computationally
in all cases that are reported in  \Cref{tab:degreesed}.

\begin{conjecture} \label{conj:betaED}
  If $m \geq 2$ then the  degree in \Cref{thm:betaED} is $\,\beta_{ED}(m,n) = m\, 2^{m}\binom{n}{m}$.
\end{conjecture}

Analogously to \Cref{prop:m2}, the above formula fails in the case $m=1$.

\begin{proposition}
\label{prop:medialaxis}
  If $m=1$ then the SDP-exact region $\mathcal{R}^{\it ed}_{\bf f}$ is dense in~$\RR^n$.
  If ${\bf f}$ is generic,
  then $\deg(\pi)=2$ and the algebraic boundary
  $\,\partial_{\rm alg}\mathcal{R}^{\it ed}_{\bf f}$
  consists of $n$~hyperplanes.
  The topological boundary $\,\partial \mathcal{R}^{\it ed}_{\bf f} = \RR^n \backslash \mathcal{R}^{\it ed}_{\bf f}$
  is contained in at most two of these $n$ hyperplanes:
  \begin{itemize}
    \item If $V_{\bf f}$ is an ellipsoid then $\partial \mathcal{R}^{\it ed}_{\bf f}$ is
 the relative interior of an ellipsoid in a hyperplane.  %\vspace{-0.07in}
     \item Otherwise, $\partial \mathcal{R}^{\it ed}_{\bf f}$ spans two hyperplanes $H_1, H_2$, and $\partial \mathcal{R}^{\it ed}_{\bf f} \cap H_i$ is bounded by a quadric. %\vspace{-0.07in}
    \item The boundary $\partial \mathcal{R}^{\it ed}_{\bf f}$ coincides with the cut locus of the quadratic hypersurface $V_{\bf f}$.
  \end{itemize}
\end{proposition}

\begin{figure}[htb]
 \centering
 \includegraphics[width=240pt]{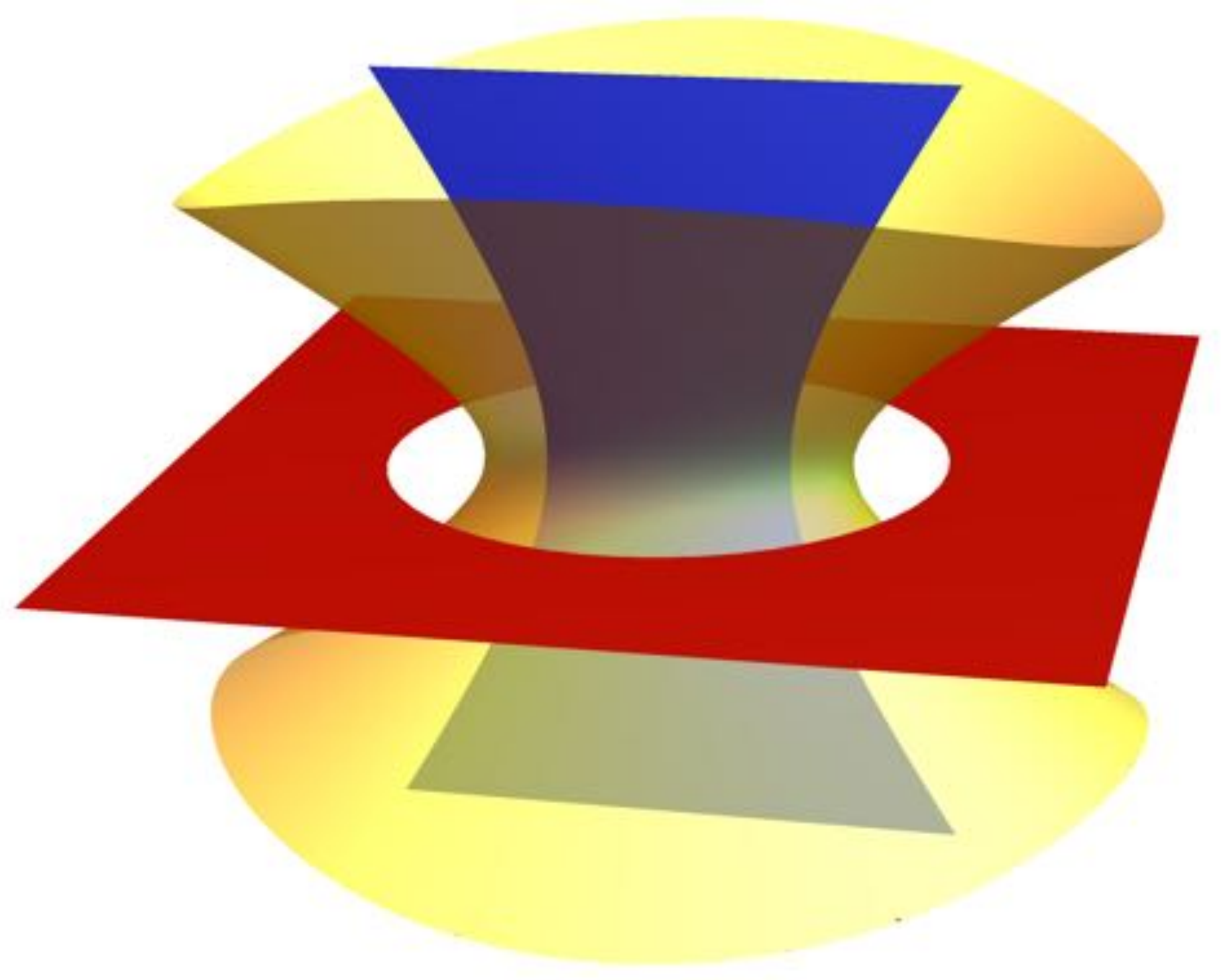}
 \caption{The cut locus of a hyperboloid (yellow) lies in  two planes.
 It is the set shown in red and blue.  The complement of the cut locus is
  the SDP-exact region for the ED~problem.}
 \label{fig:medialaxis}
\end{figure}

The {\em cut locus} of a variety $V$ in $\RR^n$ is defined as the set of all points in $\RR^n$
that have two nearest points on $V$. If $V$ is the boundary of a full-dimensional region in $\RR^n$
then the part of the cut locus that lies inside the region is referred to as the {\em medial axis}.
In Figure~\ref{fig:medialaxis}, the blue region is the medial axis. The red region is in the cut locus but not in the
medial axis.

For the varieties $V_{\bf f}$ in this paper, the cut locus is always
disjoint from the SDP-exact region $\mathcal{R}^{\it ed}_{\bf f}$.
If $m=1$ and ${\bf f}$ is generic then
these two disjoint sets cover $\RR^n$, by Proposition~\ref{prop:medialaxis}.

\begin{proof}  \Cref{prop:m2} implies that that $\mathcal{R}^{\it ed}_{\bf f}$ is dense in $\RR^n$.
We drop indices and set $\,f(x) = x^T A x + 2 a^Tx + \alpha$.
Let $\omega_1<\dots<\omega_n$ be the eigenvalues of~$A$, and let $v_i$ be the corresponding eigenvectors.
We shall assume that $\omega_1<0<\omega_n$.
The master spectrahedron is the interval
$$
  {\rm S}^{\it ed}_{\bf f}
 \, = \,\{ \lambda \in \RR \,:\, I_n - \lambda A \succ 0 \}
 \, =\, (1/\omega_1,1/\omega_m),
$$
and thus $\partial {\rm S}^{\it ed}_{\bf f} = \{1/\omega_1,1/\omega_n\}$.
Let $\lambda_i = 1/\omega_i$
and $\psi_i(x):= \psi(x,\lambda_i) = (I_n-\lambda_i A) x - \lambda_i a$.
The image of $\psi_i$ is the hyperplane $\,H_i \, = \,\{ u \in \RR^n \, :\, v_i^T u + \lambda_i v_i^T a = 0 \}$.
The fiber of $\psi_i$ over a point $u \in H_i$ is a line.
That line has a parametrization $\phi_i:\RR\to\RR^n$, $\,t \mapsto t v_i + b_u$,
where $b_u$ depends linearly on $u$.
Then $f(\phi_i(t))=0$ is a quadratic equation in $t$ with two solutions.
This proves that the morphism $\pi$ restricts to  a 2-to-1 map from $V_{\bf f}$ onto $H_i$, and thus $\deg(\pi)= 2$.
The boundary $\partial\mathcal{R}^{\it ed}_{\bf f}\cap H_i$ is given by requiring that both solutions of $f(\phi_i(t))=0$ are real.
This is the solution set to a quadratic discriminantal inequality for $u \in H_i$.
Thus $\partial \mathcal{R}^{\it ed}_{\bf f} \cap H_i$ is bounded by a quadric for $i\in \{1,n\}$.
Since the Galois group for the $n$ eigenvalues acts transitively,
 the algebraic boundary is
$\,\partial_{\rm alg} \mathcal{R}^{\it ed}_{\bf f}
= \bigcup_{i=1}^n H_i $.
\end{proof}

\begin{remark} \rm
The derivation above leads to a formula for the cut locus of an arbitrary quadratic hypersurface in $\RR^n$.
For the special case of ellipsoids, this was found by  Degen \cite{Deg}.
\end{remark}

We close this section with the analog to \Cref{thm:betaED}
for the problem {\bf (Lin)} where \eqref{eq:qp1} has
linear objective function $g$.
Now the cone $ {\rm S}^{\it lin}_{\bf f}$ on the left of \eqref{eq:masterspec} is the master spectrahedron.
The linear map \eqref{eq:affpsi} gets replaced by
$  \psi\,:\, \RR^n \times \RR^m \to \RR^n,
  (x,\lambda) \,\mapsto \,
    \textstyle\sum_{i=1}^m \lambda_i (a_i + A_i x)$.
In contrast to \eqref{eq:affpsi}, this map is now homogeneous in $\lambda$.
Hence its homogenization equals
\begin{align*} %\label{eqn:projPsilin}
  \Psi \,:\,
 \PP^n\times \PP^{m-1}\,\dashrightarrow\, \PP^{n-1}, \quad
 \bigl(\,(x_0:x)\, ,\,  \lambda\, \bigr)\, \,\mapsto \,  \,
   \textstyle\sum_{i=1}^m     \lambda_i (x_0 a_i +  A_i x).
\end{align*}
The map $\Psi$ factors as the Segre embedding $\sigma$ followed by a linear projection $\pi$:
\begin{equation*} %\label{eq:segrefollowedlin}
\PP^n \times \PP^{m-1} \,\overset{\sigma}{\longrightarrow}\,\, \PP^{(n+1)m-1}
\,\overset{\pi}{\dashrightarrow}\,\, \PP^{n-1} .
\end{equation*}
The following result transfers both \Cref{prop:EDdegree}
and \Cref{thm:betaED} to the linear problem.

\begin{theorem}
Let $\,{\bf f}$ be generic and $m \geq 2$. The algebraic degree of {\bf (Lin)} equals
$\,2^m \binom{n-1}{m-1}$.
The degree of the algebraic boundary $\,\partial_{\rm alg} \mathcal{R}^{\it lin}_{\bf f}\,$ of the
SPD-exact region equals
\begin{equation}
\label{eq:degpilin}
  \beta_{\it lin}(m,n) \quad =\quad \frac{1}{\deg(\pi)} \cdot 2^{m} n \binom{n-2}{m-2}.
\end{equation}
 \end{theorem}

\begin{proof}
The first statement is \cite[Theorem~2.2]{NR} for $d_0 =1$ and $d_1 = \cdots = d_m = 2$. The proof
of \eqref{eq:degpilin} mirrors the proof of \Cref{thm:betaED}, but with $m$ replaced by $m-1$.
The analogue to \Cref{cor:hasdimension} says that
$\,V_{\bf f} \times \partial_{\rm alg} {\rm S}^{\it lin}_{\bf f} \,$
has dimension $(n-m) + (m-2)$ and degree $2^m n \binom{n-2}{m-2}$.
\end{proof}

Just like in \Cref{conj:betaED},
  we believe that $\deg(\pi) = 1$, so that
$\beta_{\it lin}(m,n) = 2^m n \binom{n-2}{m-2}$.
There are notable differences between {\bf (Lin)} and {\bf (ED)}.
First, it is preferable to assume that $V_{\bf f}$ is compact, so that \eqref{eq:qp1} is always bounded.
Second, the SDP-exact region $\mathcal{R}^{\it lin}_{\bf f}$ is a  cone in $\RR^n$, so its algebraic boundary
$\,\partial_{\rm alg} \mathcal{R}^{\it lin}_{\bf f}\,$ should be thought of as a hypersurface in $\PP^{n-1}$.

\begin{example}[$m=2, n=3$] \rm Consider the curve
shown in the upper right of \Cref{fig:hyperboloids}.
After a projective transformation,
 $V_{\bf f} \subset \RR^3$ is bounded with two connected components.
Its theta body ${\rm TH}_1({\bf f})$ is an intersection of two solid ellipsoids that strictly contains ${\rm conv}(V_{\bf f})$.
The region $\mathcal{R}^{ lin}_{\bf f}$ consists of linear functionals whose minimum is the same for the
two convex bodies. Its algebraic boundary
$\,\partial_{\rm alg} \mathcal{R}^{\it lin}_{\bf f}\,$ is an irreducible curve in $\PP^2$ of
degree $\beta_{\it lin}(2,3) = 12$.
This is analogous to Figure \ref{fig:thetabody},
where $n=2$ and $\partial_{\rm alg} \mathcal{R}^{\it lin}_{\bf f}$ consists of $8$ points on the line $\PP^1$.
\end{example}

\section{Computing Spectrahedral Shadows}\label{s:6}

The previous section focused on the case when ${\bf f}$ is generic.
We here consider the ED~problem for overconstrained systems of quadratic equations.
These are important in many applications (e.g.,~tensor approximation, computer vision).
For a concrete example see \cite[Example~3.7]{DHOST}. These cases do not exhibit the generic
behavior. The degree computed for generic ${\bf f}$ in \Cref{thm:betaED}
serves as an upper bound for the corresponding degree when ${\bf f}$ is special.

In this section we discuss the SDP-exact region for the ED~problem when the constraints can be arbitrary equations of degree two.
We change notation by setting $m=c+p$ and
by considering a variety $V_{\bf f}$ of codimension $c$ in $\RR^{n}$
that is cut out by $c+p$ quadratic polynomials ${\bf f} = (f_1,\ldots,f_{c+p})  $ in ${x} = (x_1,\ldots,x_n)$.
If $p \geq 1$ then $V_{\bf f}$ is not a complete intersection.

Recall from \Cref{thm:unionspectrahedron}  that $\mathcal{R}^{\it ed}_{\bf f}$
is a union of spectrahedral shadows, one for each point $x \in V_{\bf f}$.
Each shadow lies in the $c$-dimensional affine space through $x$ that is normal to $V_{\bf f}$.
Thus $\mathcal{R}_{\bf f}$ is the union over an $(n-c)$-dimensional family of $c$-dimensional spectrahedral shadows.
The algebraic boundary $\partial_{\rm alg} \mathcal{R}^{\it ed}_{\bf f}$ can be written in a similar way.

By \cite[Theorem~1.1]{SS}, the {\em expected degree} of the boundary of each individual shadow is
$$ \delta(p+1,n,*) \,\,= \,\, \sum_r \delta(p+1,n,r), $$
where $r$ runs over the {\em Pataki range} of possible matrix ranks.
A key observation in \cite{SS} is that this only depends on the
codimension $p$ of the projection and not on the  dimension of the  spectrahedral
shadow. Note that the latter dimension is $c$ for regular points ${x}$ on $V_{\bf f}$.

We define the {\em expected degree} of our SDP-exact boundary
$\,\partial_{\rm alg} \mathcal{R}^{\it ed}_{\bf f}\, $ to be the product
\begin{equation}
\label{eq:expected}
 \binom{n-1}{n-c} \cdot \deg(V_{\bf f}) \cdot  \delta( p+1,n,*) .
 \end{equation}
This quantity should be an upper bound for the actual
degree of the hypersurface $\partial_{\rm alg}(\mathcal{R}^{\it ed}_{\bf f})$,
and we think that this bound should be attained in situations that are generic enough.

In what follows we present several explicit examples
of SDP-exact regions where $p \geq 1$.
We use $x = (x_1,\ldots,x_n)$ to denote points on $V_{\bf f}$
and we use $u = (u_1,\ldots,u_n)$ for points on $\partial_{\rm alg} \mathcal{R}^{ed}_{\bf f}$.
Our discussion elucidates formula \eqref{eq:expected}
and connects it to scenarios seen earlier.

\begin{example}[$n=3,c=2,p=0$]  \label{ex:twistedcubic1} \rm
  The equations $f_1= x_2 - x_1^2$ and $f_2= x_3 - x_1 x_2$ from \Cref{ex:twistedcubic0} cut out the twisted cubic curve
  $V_{\bf f}$ in $\RR^3$. The master spectrahedron
${\rm S}^{\it ed}_{\bf f}$ is the parabola $\{ \lambda \in \RR^2 : \lambda_2^2 < 2 \lambda_1 + 1 \}$.
The normal plane at the point $x=(t,t^2,t^3)$ in $V_{\bf f}$ equals
\begin{equation}
\label{eq:normalplane}
  \bigl\{ \,(u_1,u_2,u_3) \in \RR^3 \,\,:\,\,  u_1 + 2 t u_2 + 3 t^2 u_3 = 3 t^5 {+} 2 t^3 {+} t \,\bigr\}.
\end{equation}
Since $c=0$, the image ${x} - \tfrac{1}{2}{\rm Jac}_{\bf f}({x}) \cdot {\rm S}^{\it ed}_{\bf f}$
is a parabola in that plane, defined by the equation
$u_3^2 + 2 u_2 - 2 (t^3 {-} t) u_3 + t^6 {-} 2 t^4 {-} 2 t^2 {-} 1 = 0$.
Together with (\ref{eq:normalplane}) we now have two equations in four unknowns $t,u_1,u_2,u_3$.
By eliminating $t$ from these two polynomials, we~obtain
\begin{tiny} $$  \begin{matrix}
64u_2^6u_3^2+16u_1^3u_2^3u_3+408u_1^2u_2^3u_3^2-64u_1u_2^5u_3-96u_1u_2^3u_3^3+128u_2^7-256u_2^5u_3^2
-56u_2^3u_3^4+u_1^6-30u_1^5u_3-80u_1^4u_2^2+294u_1^4u_3^2 -416u_1^3u_2^2u_3  \\ -880u_1^3u_3^3
+880u_1^2u_2^4-876u_1^2u_2^2u_3^2-588u_1^2u_3^4+32u_1u_2^4u_3+256u_1u_2^2u_3^3-120u_1u_3^5-
576u_2^6+304u_2^4u_3^2+148u_2^2u_3^4-8u_3^6 +1140u_1^4u_2 \\ -1092u_1^3u_2u_3-2544u_1^2u_2^3-
558 u_1^2u_2u_3^2+192u_1u_2^3u_3-408u_1u_2u_3^3+1088u_2^5-138u_2u_3^4-2670u_1^4-600u_1^3u_3+2832
u_1^2u_2^2+207u_1^2u_3^2 +39u_3^4 \\ -96u_1u_2^2u_3   +120u_1u_3^3-1120u_2^4-228u_2^2u_3^2  -1332u_1^2
u_2-108u_1u_2u_3+680u_2^3+144u_2u_3^2+189u_1^2+54u_1u_3-244u_2^2-27u_3^2+48u_2-4.
\end{matrix}
$$ \end{tiny}
$\!\!\!\!$ This irreducible polynomial of degree $8$ defines the SDP-exact boundary $\partial_{\rm alg} \mathcal{R}^{\it ed}_{\bf f}$
around $V_{\bf f}$.
This surface and the curve $V_{\bf f}$
are shown in the left of \Cref{fig:twistedcubic}.
The surface is ruled by the parabolas in the normal bundle of the curve.
This ruling is shown on the right in \Cref{fig:twistedcubic}.
\end{example}

Our next example shows that the SDP-exact region
is \underbar{not} an invariant of   the variety $V_{\bf f}$. It depends on the
choice of defining equations.  We can have $\,V_{\bf f} = V_{\bf f'}\,$ but
 $\,\mathcal{R}^{\it ed}_{\bf f} \not = \mathcal{R}^{\it ed}_{{\bf f}'}$.

\begin{example}[$n=3,c=2,p=1$] \label{ex:321} \rm
We continue \Cref{ex:twistedcubic1} and set
$f_3 = x_1x_3-x_2^2$. Then  ${\bf f}' = (f_1,f_2,f_3)$ defines
the same twisted cubic curve as before. The master spectrahedron
${\rm S}^{\it ed}_{\bf f'}$ lives in $\RR^3$ and has degree $3$, like
the left body in \Cref{fig:somosa}.
Planar projections of such an elliptope have expected degree $\delta(2,3,*) = 6$. Here, the degree
drops to $4$ because ${\rm S}^{\it ed}_{{\bf f}'}$ is degenerate: it is singular at only two points (in $\PP^3$).
The spectrahedral shadow ${x} - \tfrac{1}{2}{\rm Jac}_{{\bf f}'}({x}) \cdot {\rm S}^{\it ed}_{{\bf f}'}$
around ${x} = (t,t^2,t^3)$ is defined by a quartic curve in the normal plane.
The SDP-exact boundary $\partial_{\rm alg} \mathcal{R}^{\it ed}_{{\bf f}'}$ is an irreducible surface of degree~$9$,
with defining polynomial
\begin{tiny} $$ \begin{matrix}
    5832 u_2^3 u_3^6+27648 u_2^6 u_3^2-62208 u_1 u_2^4 u_3^3-2916 u_1^2 u_2^2 u_3^4+15552 u_2^4 u_3^4-5832 u_1^3 u_3^5+8748 u_1^2 u_3^6-5832 u_2^2 u_3^6-4374 u_1 u_3^7+729 u_3^8-41472 u_1^2 u_2^5\\+86400 u_1^3 u_2^3 u_3+27648 u_1 u_2^5 u_3+60750 u_1^4 u_2 u_3^2-41472 u_1^2 u_2^3 u_3^2-62208 u_2^5 u_3^2-106920 u_1^3 u_2 u_3^3+85536 u_1 u_2^3 u_3^3+71442 u_1^2 u_2 u_3^4-19656 u_2^3 u_3^4\\-19440 u_1 u_2 u_3^5+3888 u_2 u_3^6-84375 u_1^6-54000 u_1^4 u_2^2+72576 u_1^2 u_2^4+202500 u_1^5 u_3-19440 u_1^3 u_2^2 u_3-48384 u_1 u_2^4 u_3-220725 u_1^4 u_3^2+6912 u_1^2 u_2^2 u_3^2\\+58032 u_2^4 u_3^2+140454 u_1^3 u_3^3-35424 u_1 u_2^2 u_3^3-54027 u_1^2 u_3^4+8424 u_2^2 u_3^4+11178 u_1 u_3^5-1161 u_3^6+40050 u_1^4 u_2-50760 u_1^2 u_2^3-21132 u_1^3 u_2 u_3\\+33840 u_1 u_2^3 u_3+11880 u_1^2 u_2 u_3^2-28744 u_2^3 u_3^2+3708 u_1 u_2 u_3^3-1314 u_2 u_3^4-7431 u_1^4+17736 u_1^2 u_2^2+6112 u_1^3 u_3-11824 u_1 u_2^2 u_3-3246 u_1^2 u_3^2\\+7976 u_2^2 u_3^2+312 u_1 u_3^3+37 u_3^4-3096 u_1^2 u_2+2064 u_1 u_2 u_3-1176 u_2 u_3^2+216 u_1^2-144 u_1 u_3+72 u_3^2.
\end{matrix}
$$ \end{tiny}
$\!\!\!\!$
The above polynomial is also the defining equation of the cut locus of the twisted cubic curve.
In fact, the SDP-exact region $\,\mathcal{R}^{\it ed}_{\mathbf{f}'}\,$ is dense in $\RR^3$ and only misses the cut locus.
This is similar to the behavior we saw in \Cref{prop:medialaxis} for quadratic hypersurfaces.
\end{example}

\begin{remark} \rm
Quadratic hypersurfaces and the twisted cubic curve share an important geometric property.
They are {\em varieties of minimal degree}.  Blekherman et al.~\cite{BSV} showed that
every non-negative quadratic form on a variety of minimal degree admits a sum-of-squares representation.
The converse holds as well.
This property implies that  $\,\mathcal{R}^{\it ed}_{\mathbf{f}}\,$ is dense in $\RR^n$
whenever ${\bf f}$ spans the full system of all quadrics vanishing on such a variety $V_{\bf f}$ in $\RR^n$.
\end{remark}

Our bundle of spectrahedral shadows is interesting even
for finite varieties ($c=n$). We demonstrate this for
point configurations in $\RR^3$. As we remove
points from the eight points in \Cref{fig:somosas},
the algebraic degree increases for the region around each remaining point.

\begin{figure}[htb]
 \centering
 \includegraphics[width=260pt]{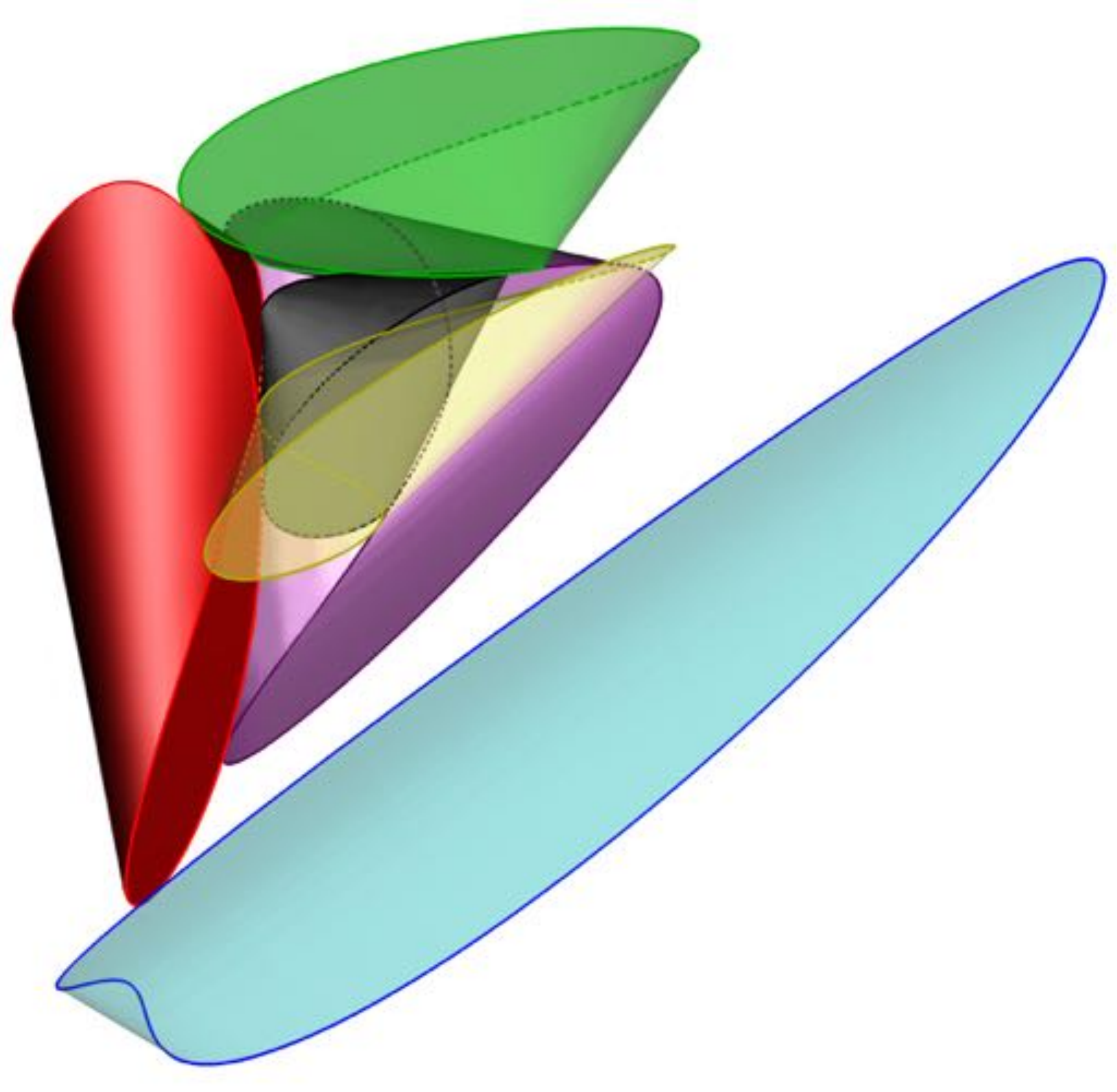}
 \caption{The SDP-exact region for the ED~problem on six points in $\RR^3$
   consists of six spectrahedral shadows. Each shadow is the convex hull of a
   highlighted curve of degree four.}
 \label{fig:tacos}
\end{figure}

\begin{example}[$n{=}3,c{=}3,p{=}1$] \label{ex:331} \rm
Six general points in $\RR^3$ are cut out by four quadrics, e.g.,
\begin{align*}
 {\bf f}  \;=\;&
 (9 x_1 x_3-5 x_2 x_3-x_3^2+x_3,\,
   6 x_2^2-13 x_2 x_3+x_3^2-6 x_2-x_3,\,\\
 &\qquad 2 x_1 x_2-6 x_1 x_3+x_2 x_3+x_3^2-x_3,
   6 x_1^2-5 x_2 x_3-x_3^2-6 x_1+x_3), \\
  V_{\bf f} \;=\;&
  \{\, (0,0,0)\,,\,(0,0,1)\,,\,(0,1,0)\,,\,(1,0,0)\,,\,(-2,-3,-2)\,,\, (-\tfrac{1}{2},-\tfrac{1}{2},-1) \,\}.
\end{align*}
The master spectrahedron ${\rm S}^{\it ed}_{\bf f}$ has degree $n=3$ and it lives in $\RR^4$.
It is the convex hull of its rank-one points, which form a rational curve of degree four.
By \cite[Example~1.3]{SS}, the projections of ${\rm S}^{\it ed}_{\bf f}$ into  $\RR^3$
are spectrahedral shadows of degree $6 = \delta(2,3,*)$,
and each shadow is the convex hull of a curve of degree four.
\Cref{fig:tacos} illustrates the six shadows.
As predicted in \eqref{eq:expected},
the SDP-exact boundary has degree $1 \cdot 6 \cdot 6 = 36 $.
\end{example}

\begin{figure}[htb]
 \centering
 \includegraphics[width=210pt]{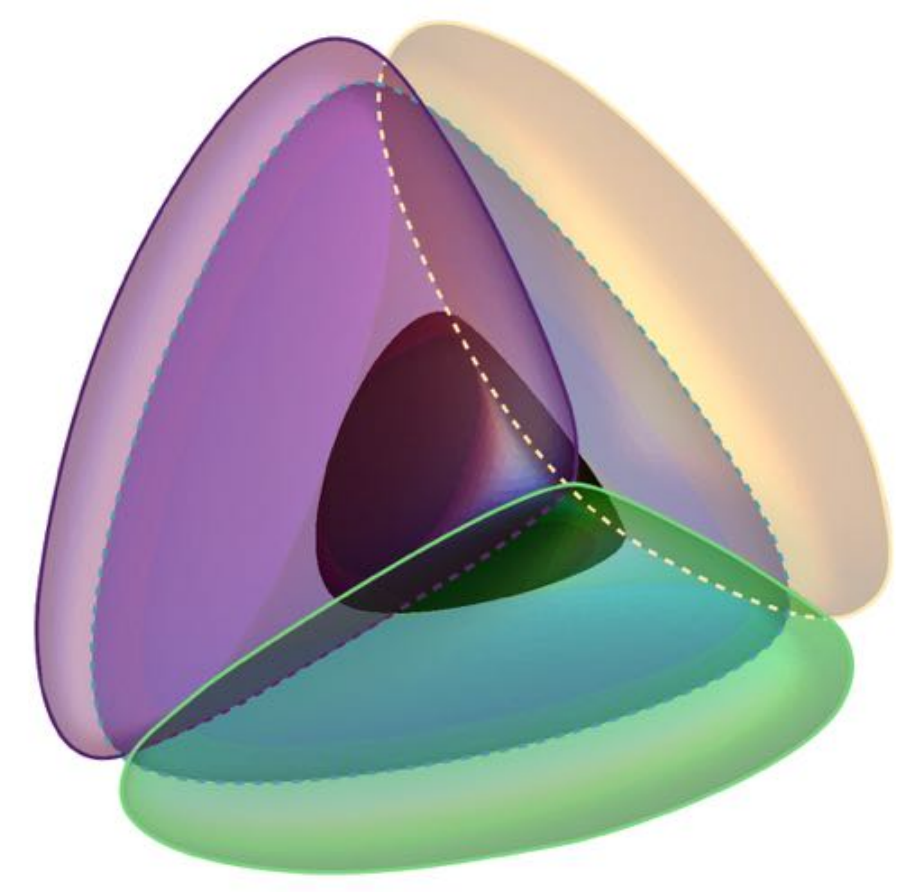}
 \caption{The SDP-exact region $\mathcal{R}^{\it ed}_{\bf f}$
  for five points in $\RR^3$ consists of five dual elliptopes.}
 \label{fig:dualsomosas}
\end{figure}

\begin{example}[$n{=}3,c{=}3,p{=}2$] \label{ex:332} \rm
Five general points in $\RR^3$ are cut out by five quadrics, e.g.,
\begin{align*}
  {\bf f} &\;=\;
  (\, x_2 x_3-x_1,\, x_1 x_3-x_2 x_3+x_1-x_2,\, x_2^2-x_3^2,\, x_1 x_2-x_3,\, x_1^2-x_3^2\,), \\
  V_{\bf f} &\;=\;
  \{ (0,0,0),\, ( 1, 1, 1),\, ( 1,-1,-1),\, (-1, 1,-1),\, (-1,-1, 1) \}.
\end{align*}
 The master spectrahedron ${\rm S}^{\it ed}_{\bf f}$ lives in $\RR^5$. It is an affine hyperplane section
 of the cone of positive semidefinite $3 \times 3$ matrices. Its projections into $\RR^3$
 look like the dual elliptope in \Cref{fig:somosa}. Such a spectrahedral
 shadow has degree $\delta(3,3,*) = 4+4$, as seen in the left box of the $p=2$ row in \cite[Table~1]{SS}.
Its boundary is given by four planes and a quartic surface.

Thus the SDP-exact region $\mathcal{R}^{\it ed}_{\bf f}$ consists of five
dual elliptopes, as seen in \Cref{fig:dualsomosas}.
They touch pairwise along their circular facets.
 For instance, the region around $(0,0,0)$ is bounded by the planes
$\{2 u_1+2 u_2-2 u_3=-3\}$, $\{2 u_1-2 u_2+2 u_3=-3\}$,  $\{2 u_1-2 u_2-2 u_3=3\}$, $\{2 u_1+2 u_2+2 u_3=3\}$,
and the quartic {\em Steiner surface}
$ \{u_1^2 u_2^2+u_1^2 u_3^2+u_2^2 u_3^2+3 u_1 u_2 u_3 = 0\}$.
Again, the prediction in \eqref{eq:expected} is correct, since
the boundary of $\mathcal{R}^{\it ed}_{\bf f}$ has degree $1 \cdot 5 \cdot (4+4) = 40 $.
\end{example}

The algebraic computation of projections of spectrahedra is  very hard
(cf.~\cite[Remark~2.3]{SS}).
In our situation, it is even harder, since
we are dealing with a family of varying projections, one for each point ${x}$ in
the variety $V_{\bf f}$. We demonstrate this in \Cref{alg:EDp1}.

Examples \ref{ex:321} and \ref{ex:331} were computed with \Cref{alg:EDp1} as is.
This works because $V_{\bf f}$ is smooth in both of these cases.
If $V_{\bf f}$ is singular then we must saturate the ideal given in step~\ref{step:manyeqns}
with respect to the ideal of $c \times c$ minors of ${\rm Jac}_{\bf f}({x})$ prior
to the elimination in step~\ref{step:elimination}.
\begin{algorithm}[H]
  \caption{Computing SDP-exact boundaries for the ED~problem (case $p=1$)}
  \label{alg:EDp1}
\setlength{\belowdisplayskip}{8pt}
\setlength{\abovedisplayskip}{8pt}
\begin{algorithmic}[1]
 \Require{Quadratic polynomials $f_1,\ldots,f_{c+1}$ defining $V_{\bf f}$ of codimension $c$ in~$\RR^{n}$.}
\Ensure{Polynomial $\psi({u}) = \psi(u_1,\ldots,u_n)$ that defines
  the algebraic boundary $\partial_{\rm alg} \mathcal{R}^{\it ed}_{\bf f}$. \smallskip }
\State Compute the Jacobian matrix ${\rm Jac}_{\bf f}({x})$ of format $n \times (c{+}1)$.
\State Compute the Lagrangian $\mathcal{L}(\lambda,{x})$ in \eqref{eq:lagrangian} and its
Hessian ${\rm H}(\lambda)$ in \eqref{eq:hessian}.
\State Let $h(\lambda) = {\rm det}({\rm H}(\lambda))$ and compute the
gradient $\nabla_\lambda(h)$, a row vector of length $c+1$.
\State Let ${\bf g}(\lambda,{x})$ be the vector of
all maximal minors of the $(n+1) \times (c+1)$ matrix
{\small $\begin{bmatrix} \nabla_\lambda (h)  \\  {\rm Jac}_{\bf f}({x}) \end{bmatrix}.$}\label{step:augmented}
\State Construct the system of equations in $(c{+}1)+ 2 n$ unknowns $(\lambda,{x},{u})$:
\begin{equation*}
 {\bf f}(x)=0, \quad {\bf g}(\lambda,x)=0, \quad h(\lambda)=0  \quad
 {\rm and} \quad {u} \, = \,{x} \,-\, \tfrac{1}{2}\,{\rm Jac}_{\bf f}({x}) \lambda.
\end{equation*}\Comment{This is expected to cut out a variety of dimension $n-1$ in $\RR^{c + 2 n + 1}$.}\label{step:manyeqns}
\State Eliminate $\lambda$ and ${x}$ from the above system
to get the desired polynomial $\psi({u})$.\label{step:elimination}
\end{algorithmic}
\end{algorithm}
\Cref{alg:EDp1} can be modified to also work when $p \geq 2$
but the details are subtle. The polynomial $h(\lambda)$ gets replaced
by the ideal of $(c{+}2{-}p) \times (c{+}2{-}p)$ minors of the matrix ${\rm H}(\lambda)$,
and the first row $\nabla_\lambda(h)$ in the
augmented Jacobian in step~\ref{step:augmented} gets replaced by the Jacobian matrix
of that determinantal ideal. This requires great care since these matrices are large.

\begin{remark} \rm
It would be interesting to study the tangency
behavior of the spectrahedral shadows in our bundles.
For instance, pairs of convex bodies meet in a point in Figure~\ref{fig:somosas},
they meet in a line segment in Figure~\ref{fig:tacos}, and they meet in a common circular facet in
Figure~\ref{fig:dualsomosas}.
\end{remark}

\noindent
{\bf Acknowledgements.}
We thank Thomas Endler, Laureano Gonz\'alez-Vega, and Kristian Ranestad for their help with this project.
Bernd Sturmfels was partially supported by the Einstein Foundation Berlin and the  US National Science Foundation.
Diego Cifuentes and Corey Harris were at the MPI-MiS Leipzig during the development of this work.

\bigskip

\begin{small}

\end{small}

\vfill
\smallskip
\bigskip
\noindent
\footnotesize {\bf Authors' addresses:}

\smallskip

\noindent Diego Cifuentes,
Massachusetts Institute of Technology
\hfill {\tt diegcif@mit.edu}

\noindent
Corey Harris, University of Oslo
\hfill {\tt coreyh@math.uio.no}

\noindent Bernd Sturmfels,
 \  MPI-MiS Leipzig and
UC  Berkeley \hfill  {\tt bernd@mis.mpg.de}
\end{document}